\documentclass[brochure,12pt,french]{smfbourbaki}
        
\usepackage[T1]{fontenc}
\usepackage{lmodern,amssymb,mathrsfs,bm,textcomp}

\usepackage[french]{babel}
\usepackage{xspace}

\def\ACF{\textsc{acf}}
\newcommand\GVF[1][]{\textsc{gvf}\textsubscript{#1}}
\def\ddc{\mathop{dd^{\mathrm c}}}
\def\C{\mathbf C}
\def\F{\mathbf F}
\def\N{\mathbf N}
\def\P{\mathbf P}
\def\Q{\mathbf Q}
\def\R{\mathbf R}
\def\Z{\mathbf Z}
\def\ovR{\overline\R}
\def\an{{\mathrm{an}}}
\def\Aut{\mathrm{Aut}}
\def\Val{\mathrm{Val}}
\def\PVal{\mathrm{PVal}}
\def\QVal{\mathrm{QVal}}
\def\hDiv{\widehat{\mathrm{Div}}}
\def\Div{\mathrm{Div}}
\def\bDiv{\mathrm{bDiv}}
\def\div{\operatorname{div}}
\def\hdeg{\widehat{\operatorname{deg}}}
\def\hdiv{\widehat{\operatorname{div}}}
\def\degtr{\operatorname{deg.tr}\nolimits}
\def\abs #1{\lvert #1\rvert}
\def\Abs #1{\left\lvert #1\right\rvert}
\def\norm #1{\lVert #1\rVert}
\let\epsilon\varepsilon\let\eps\varepsilon
\let\phi\varphi
\def\vol{\operatorname{vol}}
\def\hvol{\widehat{\operatorname{vol}}}
\def\Spec{\mathrm{Spec}}
\def\Card{\mathrm{Card}}
\def\sozat{\,;\,}
\def\Ht{\operatorname{ht}}
\def\setminus{\mathchoice
    {\mathbin{\vrule height .72ex width 1.61ex depth -.38ex}}% 12
    {\mathbin{\vrule height .72ex width 1.61ex depth -.38ex}}% 12
    {\mathbin{\vrule height .50ex width 0.85ex depth -.28ex}}%  9
    {\mathbin{\vrule height .20ex width 0.570ex depth -.24ex}}%  7
}

\makeatletter
  \def\cl@section{\@elt {subsection}}
  \def\cl@subsection{\@elt {subsubsection}\@elt {defi}}
  
  \let\c@equation\c@defi

\makeatother

  \DeclareMathSymbol{A}{\mathalpha}{operators}{`A}%
  \DeclareMathSymbol{B}{\mathalpha}{operators}{`B}%
  \DeclareMathSymbol{C}{\mathalpha}{operators}{`C}%
  \DeclareMathSymbol{D}{\mathalpha}{operators}{`D}%
  \DeclareMathSymbol{E}{\mathalpha}{operators}{`E}%
  \DeclareMathSymbol{F}{\mathalpha}{operators}{`F}%
  \DeclareMathSymbol{G}{\mathalpha}{operators}{`G}%
  \DeclareMathSymbol{H}{\mathalpha}{operators}{`H}%
  \DeclareMathSymbol{I}{\mathalpha}{operators}{`I}%
  \DeclareMathSymbol{J}{\mathalpha}{operators}{`J}%
  \DeclareMathSymbol{K}{\mathalpha}{operators}{`K}%
  \DeclareMathSymbol{L}{\mathalpha}{operators}{`L}%
  \DeclareMathSymbol{M}{\mathalpha}{operators}{`M}%
  \DeclareMathSymbol{N}{\mathalpha}{operators}{`N}%
  \DeclareMathSymbol{O}{\mathalpha}{operators}{`O}%
  \DeclareMathSymbol{P}{\mathalpha}{operators}{`P}%
  \DeclareMathSymbol{Q}{\mathalpha}{operators}{`Q}%
  \DeclareMathSymbol{R}{\mathalpha}{operators}{`R}%
  \DeclareMathSymbol{S}{\mathalpha}{operators}{`S}%
  \DeclareMathSymbol{T}{\mathalpha}{operators}{`T}%
  \DeclareMathSymbol{U}{\mathalpha}{operators}{`U}%
  \DeclareMathSymbol{V}{\mathalpha}{operators}{`V}%
  \DeclareMathSymbol{W}{\mathalpha}{operators}{`W}%
  \DeclareMathSymbol{X}{\mathalpha}{operators}{`X}%
  \DeclareMathSymbol{Y}{\mathalpha}{operators}{`Y}%
  \DeclareMathSymbol{Z}{\mathalpha}{operators}{`Z}%

\usepackage{enumerate}

\usepackage[bookmarksdepth=2,
    colorlinks=true, linkcolor=blue, citecolor=red, urlcolor=blue]{hyperref}
\usepackage[
backend=biber,
style=authoryear, 
citestyle=authoryear-comp,
maxnames=7,
giveninits=true,
sortcites=false %%%% to keep the order in \cite command
]{biblatex}
\usepackage{csquotes}

\newtoggle{tnbcbx@howcited}
\DeclareEntryOption[boolean]{howcited}[true]{\settoggle{tnbcbx@howcited}{#1}}
\DeclareBibliographyOption[boolean]{howcited}[true]{\settoggle{tnbcbx@howcited}{#1}}
\DeclareTypeOption[boolean]{howcited}[true]{\settoggle{tnbcbx@howcited}{#1}}

\newbibmacro{howcited}{%
  \iftoggle{tnbcbx@howcited}
    {\iffieldundef{shorthand}
       {}
       {\setunit{\addspace}%
        \printtext[parens]{%
          \bibstring{citedas}%
          \setunit{\addcolon\space}%
          \printfield{shorthand}%
          \setunit{\addslash}%
          }}}
    {}}

\renewbibmacro{finentry}{\usebibmacro{howcited}\finentry}

\DefineBibliographyStrings{french}{citedas    = {cit\'e ci-dessus avec l'acronyme}}
\DefineBibliographyStrings{english}{citedas    = {heretofore cited as}}

\DeclareDelimFormat{nameyeardelim}{\addcomma\space}

\DeclareNameAlias{sortname}{given-family}

\renewcommand{\bibnamedash}{\leavevmode\raise3pt\hbox to3em{\hrulefill}\space}

\AtEveryBibitem{%
  \clearfield{month}\clearfield{day} % Remove month, day
  \clearfield{issn} % Remove issn
  \clearfield{isbn} % Remove isbn
  \clearfield{doi} % Remove doi
  \clearname{urlyear}\clearname{urlmonth}\clearname{urlday}
  \clearlist{language} %%% remove language
  \ifentrytype{online}{}{% Remove url except for @online
  \ifentrytype{unpublished}{}{% ou bien @unpublished
    \clearfield{url}%
    \clearfield{urlyear}\clearfield{urlmonth}\clearfield{urlday}}}
}

\renewbibmacro{in:}{%
    \ifentrytype{article}{}{\printtext{\bibstring{in}\intitlepunct}}}

\DeclareFieldFormat[article,periodical,inreference]{number}{\mkbibparens{#1}}
\DeclareFieldFormat[article,periodical,inreference]{volume}{\mkbibbold{#1}}
\renewbibmacro*{volume+number+eid}{%
    \printfield{volume}%
    \setunit*{\addthinspace}% NEW
    \printfield{number}%
    \setunit{\addcomma\space}%
    \printfield{eid}}

\DeclareFieldFormat[article,inbook,incollection]{title}{\enquote{#1}\addcomma} 

\date{Juin 2025}
\bbkannee{77\textsuperscript{e} ann\'ee, 2024--2025}  % Commandes sp\'eciales
\bbknumero{1240}                                    % Commandes sp\'eciales

\title{La logique continue des corps globalement valu\'es}

\author{Antoine Chambert-Loir}
\address{Universit\'e Paris Cit\'e, IMJ-PRG \\ 8 place Aur\'elie Nemours, 75013 Paris}
\email{antoine.chambert-loir@u-paris.fr}
\begin{abstract}
Un corps globalement valu\'e est un corps muni d'une famille de valeurs
absolues satisfaisant \`a une formule du produit. Les corps de nombres ou
les corps de fonctions d'une variable fournissent des exemples
classiques, et fondamentaux, d'une telle structure alg\'ebrique ; la
th\'eorie de Nevanlinna permet de construire une telle structure sur le
corps des fonctions m\'eromorphes sur $\mathbf C$. Ces corps globalement
valu\'es peuvent \^etre abord\'es dans le cadre de la logique continue (pour
laquelle les pr\'edicats sont \`a valeurs r\'eelles), et une telle \'etude a \'et\'e
entreprise par Ben Yaacov et Hrushovski il y a presque 10 ans,
fournissant un cadre mod\`ele-th\'eorique pour la th\'eorie diophantienne des
hauteurs. Un des premiers r\'esultats fondamentaux de la th\'eorie affirme
que le corps des nombres alg\'ebriques, avec sa structure
(essentiellement unique) de corps globalement valu\'e, est
existentiellement clos : tout syst\`eme d'\'egalit\'es et in\'egalit\'es
polynomiales et d'in\'egalit\'es strictes entre hauteurs poss\`ede une
solution en nombres alg\'ebriques, pourvu qu'il en poss\`ede une dans une
extension globalement valu\'ee. La d\'emonstration, due \`a Szachniewicz,
s'inspire de celle propos\'ee par Ben Yaacov et Hrushovski dans le cas des
corps de fonctions: alors que cette derni\`ere utilisait de mani\`ere
cruciale la description par Boucksom, Demailly, P\u aun et Peternell du
c\^one des courbes mobiles dans une vari\'et\'e projective complexe, le cas
des corps de nombres repose sur des r\'esultats r\'ecents de th\'eorie
d'Arakelov.
\end{abstract}

\alttitle{The continuous logic of globally valued fields}

\begin{altabstract}
A globally valued field is a field endowed with a family of absolute
values that satisfy a product formula.  Number fields and function
fields in one variable give classical and fundamental examples;
Nevanlinna theory also gives rise to such structures on the field of
meromorphic functions on $\mathbf C$.  These globally valued fields
can be studied in the context of continuous logic (for which the
predicates are real valued), and such a study has been undertaken some
10 years ago by Ben Yaacov and Hrushovski, thus providing a
model-theoretic framework for the diophantine theory of heights.  One
of the first fundamental results in the tehory states the the field of
algebraic numbers, with its essentially unique structure of a globally
valued field, is existentially closed: every system involving
polynomial equalities and inequalities, as well as strict inequalities
in heights, possesses a solution in algebraic numbers as soon as it
possesses some solution in a globally valued extension. The proof, due
to Szachniewicz, is inspired by the proof proposed by Ben Yaacov and
Hrushovski in the case of function fields: the latter used in a
crucial way the description by Boucksom, Demailly, P\u aun and Peternell
of the cone of mobile curves in a complex projective variety, the case
of number fields relies on recent results in Arakelov geometry.
\end{altabstract}
\begin{document}

\maketitle

{\raggedleft
\itshape \`A la m\'emoire de Zo\'e Chatzidakis\par}

\section*{Introduction}
\phantomsection

La logique math\'ematique met en jeu des structures
math\'ematiques c'est-\`a-dire des ensembles, des fonctions, ainsi
que des pr\'edicats \`a valeur vrai/faux, assujettis \`a satisfaire
aux axiomes relatifs \`a ces structures.
Un de ses objectifs est d'\'elucider 
la g\'eom\'etrie des \emph{ensembles d\'efinissables}
que ces fonctions et pr\'edicats permettent de construire,
un autre est d'\'etudier les \emph{mod\`eles}
de ces axiomes, c'est-\`a-dire de structures math\'ematiques
o\`u ils sont v\'erifi\'es.
Elle pr\^ete \'egalement une attention particuli\`ere 
\`a la nature des \emph{formules} consid\'er\'ees, 
par exemple la pr\'esence ou la position des quantificateurs logiques.
En retour, elle fournit des th\'eor\`emes g\'en\'eraux (comme le th\'eor\`eme
de compacit\'e) et des concepts structuraux
(diverses notions de dimension, le concept d'imaginaire,
la th\'eorie de la stabilit\'e\dots) 
qui permettent d'aborder
des th\'eories math\'ematiques de fa\c{c}on originale.
N\'ee des travaux de Tarski et Robinson sur la th\'eorie des corps
alg\'ebriquement clos et des corps r\'eels clos, la th\'eorie des mod\`eles
s'est ouverte de fa\c{c}on f\'econde \`a des th\'eories moins \'el\'ementaires : 
alg\`ebre diff\'erentielle,
corps aux diff\'erences, corps valu\'es\dots

Cet expos\'e est consacr\'e au point de vue qu'offre la logique math\'ematique
sur la notion de \emph{hauteur} en g\'eom\'etrie arithm\'etique.
Sous sa forme la plus \'el\'ementaire, o\`u l'on s'int\'eresse
aux solutions enti\`eres d'\'equations polynomiales,
la hauteur correspond tout simplement \`a leur taille, 
et l'int\'er\^et de la consid\'erer appara\^{\i}t dans la m\'ethode de descente infinie 
explicit\'ee par \textcite{Fermat-1894} dans sa lettre d'ao\^ut 1659 \`a Carcavi.
Exploit\'ee par~\textcite{Mordell-1922} et~\textcite{Weil-1928},
puis \textcite{Northcott-1950}, 
c'est aujourd'hui un outil primordial de la g\'eom\'etrie diophantienne,
et \'egalement un objet de nombreuses questions ouvertes.
Cette notion admet aussi un analogue sur les corps 
de fonctions, au parfum plus g\'eom\'etrique, 
qu'en retour, la g\'eom\'etrie d'Arakelov se propose de diffuser
\`a la th\'eorie \og na\"{\i}ve\fg des hauteurs en arithm\'etique.

M\^eme sur les corps de fonctions,
les fonctions hauteurs sont a priori \`a valeurs r\'eelles.
Deux options s'offrent \`a la logique math\'ematique pour 
axiomatiser ces fonctions : les consid\'erer comme 
\`a valeurs dans un corps ordonn\'e, \'eventuellement r\'eel clos,
ou insister pour qu'elles soient \`a valeurs r\'eelles.
Dans le premier cas,
le th\'eor\`eme de compacit\'e et la construction de mod\`eles \og non standard \fg
en logique classique conduiront \`a des hauteurs \`a valeurs non standard
dans des corps non archim\'ediens, qui font perdre de vue l'intuition initiale
de la descente infinie. 
C'est donc la seconde option qu'on va suivre,
gr\^ace au point de vue offert par la \emph{logique continue}.
Une premi\`ere th\'eorie
avait \'et\'e propos\'ee dans~\parencite{ChangKeisler-1966}, mais la variante
r\'ecente de \parencite{BenYaacovBerensteinHensonEtAl-2008, BenYaacov-2008}
semble plus adapt\'ee.
Il est impossible de r\'esumer en deux phrases ces th\'eories 
mais deux images peuvent \^etre \'eclairantes. La premi\`ere consiste \`a penser
les nombres r\'eels comme des valeurs de v\'erit\'e, de sorte que
ces fonctions \`a valeurs r\'eelles apparaissent comme des pr\'edicats. 
La seconde, plus technique, observe que la topologie des
espaces de types de la logique math\'ematique classique 
sont compacts et totalement discontinus, 
et la logique continue donne lieu \`a des espaces localement compacts g\'en\'eraux.

Avant m\^eme toute logique continue, le premier chapitre de cet expos\'e
consistera \`a expliquer la notion de \emph{corps globalement valu\'e}
qui est l'\emph{axiomatisation} des hauteurs
propos\'ee par~\textcite{BenYaacovHrushovski-2022}
et d\'evelopp\'ee dans 
l'article~\parencite{BenYaacovDesticHrushovskiEtAl-2024}.
Il s'agit bien ici d'\'elucider les propri\'et\'es g\'en\'erales de ces fonctions
hauteurs qui justifieraient de les \'etudier sur un corps arbitraire.
De fait, la notion de corps globalement valu\'e peut \^etre abord\'ee
de plusieurs fa\c{c}ons, et tous ces regards se compl\`etent.
Partant d'une notion de \emph{hauteur} (d\'efinition~\ref{defi.h}) 
sur un corps~$F$,
nous d\'efinirons ensuite un $\Q$-espace vectoriel  r\'eticul\'e de \emph{diviseurs}
sur~$F$ et une forme lin\'eaire positive sur cet espace vectoriel,
puis une mesure sur espace de valeurs absolues g\'en\'eralis\'ees
donnant lieu \`a une formule du produit, d'o\`u la terminologie
de \emph{corps globalement valu\'e}.

Le r\'esultat principal
de ce chapitre est l'\'equivalence de ces points de vue.
Comme on peut le deviner, la pr\'esentation en termes de mesures
recoupe d'autres axiomatisations de la notion de hauteur,
comme les \og $M$-fields\fg de \textcite{Gubler-1997}
ou les \og courbes ad\'eliques\fg de \textcite{ChenMoriwaki-2020}.
Nous verrons aussi comment ces corps globalement valu\'es donnent une \'epaisseur
suppl\'ementaire \`a l'analogie entre approximation
diophantienne et th\'eorie de Nevanlinna propos\'ee par~\textcite{Vojta-1987}.

Dans un second chapitre, nous expliquons comment
l'on peut d\'ecrire les diff\'erentes structures de corps
globalement valu\'e dans le cas des corps~$F$ qui sont de type fini
sur un sous-corps~$k$.
Lorsqu'on impose que la structure soit \og triviale\fg 
sur ce sous-corps~$k$, le groupe des diviseurs \'evoqu\'es ci-dessus
s'identifie au groupe des diviseurs birationnels de~\textcite{Shokurov-1996}
sur un $k$-sch\'ema propre, int\`egre, normal de corps des fractions~$F$.
La description des structures de corps globalement valu\'e
est un raffinement  du th\'eor\`eme de~\textcite{BoucksomDemaillyPaunEtAl-2013}
qui d\'ecrit le c\^one pseudo-effectif d'une vari\'et\'e complexe
comme le dual du c\^one des courbes mobiles.

En particulier, les corps de fonctions d'une variable sur~$k$ poss\`edent
essentiellement une seule structure de corps globalement valu\'e
qui soit triviale sur~$k$. 

De m\^eme, les corps de nombres poss\`edent essentiellement une seule structure
de corps globalement valu\'e, donn\'ee par la th\'eorie classique des hauteurs.
Des d\'eveloppements r\'ecents en g\'eom\'etrie d'Arakelov permettent une description 
des structures de corps globalement valu\'e sur un corps~$F$
qui est de type fini sur~$\Q$ analogue \`a celle que nous avons
\'evoqu\'ee dans le cas g\'eom\'etrique. 

Dans un dernier chapitre, nous pr\'esentons la logique continue
des corps globalement valu\'es et expliquons les deux th\'eor\`emes
qui ont suscit\'e cet expos\'e.
Le premier, d\^u \`a~\textcite{BenYaacovHrushovski-2022}, affirme que
si $k$ est un corps, la cl\^oture alg\'ebrique~$K$ de $k(T)$,
munie de sa structure naturelle
(triviale sur~$k$) est  
un corps globalement valu\'e \emph{existentiellement clos}. 
Cela signifie en particulier que tout syst\`eme d'\'equations constitu\'e
d'\'equations et d'in\'equations polynomiales
\`a coefficients dans~$K$
qui admet une solution dans un corps globalement valu\'e~$L$ prolongeant~$K$
poss\`ede \'egalement une solution dans~$K$ dont les hauteurs seront arbitrairement
proches de la premi\`ere.
Le second th\'eor\`eme, du \`a~\textcite{Szachniewicz-2023},
est l'analogue pour le corps $\overline\Q$ des nombres alg\'ebriques.

C'est toute une th\'eorie qu'ouvrent ces deux th\'eor\`emes avec, pour l'instant, 
plus de questions que de r\'eponses,
et nous en dirons quelques mots dans un paragraphe final.

\medskip

Je remercie E.~Hrushovski et S.~Rideau d'avoir r\'epondu avec
tant de gentillesse \`a mes nombreuses questions, et les collaborateurs
de N.~Bourbaki de leur \oe il vigilant et n\'eanmoins amical.

\section{Quatre regards sur les corps globalement valu\'es}

% (1) Hauteurs
% (2) Termes locaux
% (3) Fonctionnelles divisorielles
% (4) Mesures

\textcite{BenYaacovDesticHrushovskiEtAl-2024} proposent sept regards 
sur la notion de corps globalement valu\'e. Nous en pr\'esentons
ici quatre : hauteurs, termes locaux, fonctionnelles divisorielles
et mesures. Nous expliquerons ensuite au \S\ref{ss.gvf}
comment elles conduisent \`a des notions \'equivalentes.

\subsection{Hauteurs}

Le premier de ces regards est une axiomatisation de la notion de hauteur
en g\'eom\'etrie diophantienne.

\begin{defi}\label{defi.h}
Soit $F$ un corps. Une \emph{hauteur} sur~$F$  est
une suite d'applications $h_n\colon F^n\to \R\cup\{-\infty\}$,
pour $n\in\N$,
qui v\'erifient les propri\'et\'es suivantes:
\begin{enumerate}[\upshape (i)]
\item Pour tout~$n$, on a $h_n(x)=-\infty$ si et seulement si $x=0$;
\item Pour tout $n\geq 1$, on a $h_n(1,\dots,1)=0$;
\item Pour toute permutation $\sigma\in\mathfrak S_n$
et tout $x\in F^n$, on a $h_n(x_{\sigma(1)},\dots,x_{\sigma(n)})=h_n(x)$;
\item Pour tout $x\in F^m$ et tout $y\in F^n$, on a $h_m(x)\leq h_{m+n}(x,y)$;
\item Pour tout $x\in F^m$ et tout $y\in F^n$, on a 
$h_{mn}(x\otimes y)=h_m(x)+h_n(y)$, o\`u $x\otimes y=(x_i y_j)$;
\item Il existe un nombre r\'eel~$e$ tel que
$h_n(x+y)\leq h_{2n}(x,y)+e$ pour tout $n\in\N$ et tous $x,y\in F^n$.
\end{enumerate}
\end{defi}

On notera souvent~$h$ au lieu de~$h_n$.

On dit qu'une hauteur sur~$F$ est \emph{globale}
si $h_n(c\cdot x)=h_n(x)$ 
pour tout $n\in\N$, tout $c\in F^\times $ et tout $x\in F^n$.
Elle d\'efinit alors des fonctions fonctions de $\P_{n-1}(F)$ dans~$\R$,
\'egalement not\'ees~$h$,
qui v\'erifient les propri\'et\'es alg\'ebriques classiques des \emph{hauteurs} 
sur un corps global, \`a l'exception importante du th\'eor\`eme de finitude
de Northcott.
Pour $x=(x_1,\dots,x_n)\in F^n$, on a $h_n(x)=h([1:x_1:\dots:x_n])$.

Le nombre r\'eel~$e$ dans la condition~(vi) est n\'ecessairement positif
ou nul ; il refl\`ete la possibilit\'e de places archim\'ediennes.

\begin{exem}
La th\'eorie classique des hauteurs en g\'eom\'etrie arithm\'etique
munit le corps~$\Q$ d'une fonction hauteur.
Pour tout $n\in\N$ et tout $(x_1,\dots,x_n)\in\Q^n$, 
on pose $h_n(x_1,\dots,x_n)=-\infty$ si les~$x_i$ sont tous nuls;
sinon, on note $d$ le g\'en\'erateur strictement positif
du sous-groupe ab\'elien $\langle x_1,\dots,x_n\rangle$ de~$\Q$
et l'on pose
\[ h_n(x_1,\dots,x_n) = \sup ( \log \abs{x_1},\dots, \log \abs{x_n})/d. \]
La v\'erification des cinq premiers axiomes de la d\'efinition~\ref{defi.h} 
est \'el\'ementaire, de m\^eme que le fait qu'il s'agit d'une hauteur globale.
Pour le sixi\`eme, on peut prendre $e=\log(2)$: on se ram\`ene \`a prouver
l'in\'egalit\'e~(vi) lorsque $x, y$ sont des \'el\'ements non nuls de~$\Z^n$;
alors, 
\begin{align*}
 h_n(x+y) & \leq \log \sup (\abs{x_1+y_1},\dots,\abs{x_n+y_n}) \\
& \leq \log \sup (2\sup(\abs{x_1},\abs{y_1}), \dots ) \\
& \leq \log(2) + \log \sup (\abs{x_1},\dots,\abs{x_n}, \abs{y_1},\dots,\abs{y_n}).
\end{align*}

Soit $k$ un corps. En exploitant le fait que~$\Z$ et~$k[T]$
sont tous deux des anneaux principaux, et l'analogie
entre valeur absolue des entiers et degr\'e des polyn\^omes,
on d\'efinit de m\^eme une hauteur sur le corps~$k(T)$.
Pour tout $n\in\N$ et tout $(x_1,\dots,x_n)\in k(T)^n$, 
on pose en effet $h_n(x_1,\dots,x_n)=-\infty$ si les~$x_i$ sont tous nuls;
sinon, on choisit un g\'en\'erateur 
du sous-$k[T]$-module $\langle x_1,\dots,x_n\rangle$ de~$k(T)$
et l'on pose
\[ h_n(x_1,\dots,x_n) = \sup ( \deg(x_1/d),\dots, \deg(x_n/d)). \]
On obtient ainsi une hauteur globale sur~$k(T)$; pour l'axiome~(vi), 
on peut prendre $e=0$.

Nous dirons que ce sont les structures de hauteurs \emph{standard}
sur~$\Q$ et~$k(T)$.

La th\'eorie classique des hauteurs \'etend ces d\'efinitions
aux corps de nombres et aux corps de fonctions d'une variable, 
mais il est n\'ecessaire, pour leur d\'efinition,
de recourir \`a la classification des valeurs
absolues sur ces corps. Elles apparaissent alors comme
un cas particulier de l'exemple suivant.
Nous y reviendrons \'egalement dans l'exemple~\ref{exem-aw}.
\end{exem}

\begin{exem}
La notion de $M$-corps introduite
par \textcite{Gubler-1997} ou celle de \emph{courbe ad\'elique}
d\'egag\'ee par~\textcite{ChenMoriwaki-2020} fournissent un cadre
dans lequel interpr\'eter de fa\c{c}on uniforme la th\'eorie classique
des hauteurs. Une courbe
ad\'elique est form\'ee d'un espace mesur\'e~$(\Omega,\mu)$ 
et d'une application de~$\Omega$ dans l'ensemble~$M_F$
des valeurs absolues sur un corps~$F$, not\'ee $\omega\mapsto\abs\cdot_\omega$,
telle que pour tout $x\in F^\times$,
l'application $\omega\mapsto \log \abs{x}_\omega$ 
soit int\'egrable. La formule
\begin{equation}
 h_n(x_1,\dots,x_n) = \int_\Omega \sup (\log \abs{x_1}_\omega,\dots,
\log \abs{x_n}_\omega)\, d\mu(\omega)
\end{equation}
d\'efinit alors une hauteur sur~$F$; cette hauteur est globale
si et seulement si l'on a la \emph{formule du produit}
\begin{equation}
\int_\Omega \log \abs x_\omega\, d\mu(\omega) =0 
\end{equation}
pour tout $x\in F^\times$.
\end{exem}

Ces exemples justifient la terminologie de corps globalement valu\'e
et sugg\`erent un deuxi\`eme regard sur cette structure,
en termes d'espaces mesur\'es.

\subsection{Pseudo-valuations}
Pour repr\'esenter la hauteur comme une int\'egrale sur un espace intrins\`eque,
l'axiomatisation
propos\'ee par~\textcite{BenYaacovDesticHrushovskiEtAl-2024}
conduit \`a consid\'erer des objets un peu plus g\'en\'eraux
que des valeurs absolues: les pseudo-valuations.

Si $\Gamma$ est un groupe ordonn\'e dont la loi est not\'ee additivement
\footnote{La th\'eorie des valeurs absolues
et celle des valuations conduisent \`a un double conflit 
concernant le sens de la relation d'ordre
et le caract\`ere additif du groupe de valeurs, conflit mat\'erialis\'e par la
relation \og $v(x)=\log \abs x^{-1}$\fg entre les deux notions.
Si le formalisme de~\textcite{Huber-2013} corrige ces deux conflits,
il n'est pas non plus adapt\'e \`a la g\'eom\'etrie diophantienne
dans laquelle c'est la fonction $\log \abs{x}^{-1}$ 
qui poss\`ede les propri\'et\'es de \emph{positivit\'e} appropri\'ees 
et qu'il s'agit de \emph{sommer} ou d'int\'egrer.},
on note $\overline\Gamma$ l'ensemble ordonn\'e obtenu par adjonction
d'un plus petit et d'un plus grand \'el\'ements, 
respectivement not\'es $-\infty$ et $+\infty$.
On \'etend la loi de groupe de~$\Gamma$ en une loi partiellement d\'efinie
sur $\overline\Gamma$  en posant
$(+\infty)+a=a+(+\infty)=+\infty$ pour tout $a\neq-\infty$,
et $(-\infty)+a=a+(-\infty)=-\infty$ pour tout $a\neq +\infty$.

\begin{defi}
Soit $A$ un anneau. Une pseudo-valuation sur~$A$
est une application $v\colon A\to\overline\Gamma$ v\'erifiant les propri\'et\'es suivantes:
\begin{enumerate}[\upshape (i)]
\item On a $v(0)=+\infty$ et $v(1)= 0$;
\item On a $v(xy)=v(x)+v(y)$
pour tous $x,y\in A$ tels que $\{v(x),v(y)\}\neq\{-\infty;+\infty\}$;
\item Il existe $e\in\Gamma$ tel que $v(x+y)\geq \inf(v(x), v(y))-e$
pour tous $x,y\in A$.
\end{enumerate}
\end{defi}

Dans la condition~(iii), on a n\'ecessairement $e\geq0$.
Lorsqu'on peut prendre $e=0$, la pseudo-valuation est dite 
\emph{ultram\'etrique}. 
L'introduction de ce param\`etre permet de tenir compte
de fa\c{c}on uniforme de l'existence des valeurs absolues 
au sens de~\textcite[\S9]{Weil-1951} et \'egalement
\'etudi\'ees r\'ecemment sous le nom de \emph{pseudo-valeur absolue}
par~\textcite{Sedillot-2024},
pour qui elles peuvent prendre la valeur~$+\infty$; 
alors $\log\abs{\,\cdot\,}^{-1}$ est une pseudo-valuation.
En fait, si $v$ n'est pas ultram\'etrique,
alors il existe une pseudo-semi-valeur absolue $\abs{\,\cdot\,}$
sur~$A$ telle que $v=\log\abs{\,\cdot\,}^{-1}$.
En particulier, on peut prendre $e=-\inf(0,v(2))$.

Lorsque $v$ ne prend pas la valeur~$-\infty$, on dira
que c'est une \emph{valuation}.
Lorsque $v$ est ultram\'etrique, on retrouve effectivement les valuations
au sens classique, mais conduit \'egalement \`a appeler
valuation toute fonction de la forme $\log\abs{\, \cdot\,}^{-1}$,
o\`u $\abs{\, \cdot\,}^{-1}$ est une valeur absolue (au sens classique).

Dans le cas g\'en\'eral, l'ensemble~$A^\circ$
des \'el\'ements $a\in A$ tels que $v(a)\neq -\infty$ 
est un sous-anneau de~$A$ sur lequel $v$ d\'efinit une valuation
(au sens de ce texte);
en quelque sorte, 
les pseudo-valuations sont aux valuations ce que les places
\parencite[chapitre~VI, \S2, \no2]{Bourbaki-2006d}
sont aux morphismes d'anneaux.

On dit qu'une pseudo-valuation est \emph{triviale} si elle
ne prend que les valeurs~$-\infty$, $0$ et~$+\infty$.

Dans la suite de ce texte, on s'int\'eressera essentiellement aux
pseudo-valuations \emph{non triviales} \`a valeurs r\'eelles sur un corps~$F$.

On munit l'espace $\ovR$ de sa topologie usuelle
et l'espace~$\ovR^F$ de la topologie de la convergence simple;
ils sont compacts.
On note $\Val(F)$ et $\PVal(F)$ les sous-espaces de~$\ovR^F$
form\'es des valuations et des pseudo-valuations,
ainsi que $\Val(F)_*$ et $\PVal(F)_*$ les sous-espaces
des valuations et pseudo-valuations non triviales.
Pour des raisons topologiques dont l'int\'er\^et appara\^{\i}tra plus tard, 
on introduit \'egalement la partie $\QVal(F)$ de~$\ovR^F$ d\'efinie comme
l'adh\'erence de $\PVal(F)_*$ priv\'ee de l'ensemble (ferm\'e)
des fonctions ne prenant
que les valeurs~$-\infty,0,+\infty$; elle est localement compacte
et $\PVal(F)_*$ en est un sous-espace dense.

\subsection{Termes locaux}

On appelle \emph{polyn\^ome tropical} tout terme
dans le langage des $\Q$-espaces vectoriels ordonn\'es.
Autrement dit, un polyn\^ome tropical est d\'efini par r\'ecurrence
\`a partir des constructions suivantes:
\begin{itemize}
\item Tout symbole de variable d\'efinit un polyn\^ome tropical;
\item Si $t$ et $t'$ sont des polyn\^omes tropicaux,
alors $\sup(t,t')$ et $t+t'$ sont des polyn\^omes tropicaux;
\item Si $t$ est un polyn\^ome tropical et $c$ est un nombre
rationnel, $c\cdot t$ est un polyn\^ome tropical.
\end{itemize}

Si $V$ est un ensemble de symboles de variables,
on note $\mathscr T_V$ l'ensemble des polyn\^omes tropicaux
dont les variables appartiennent \`a~$V$.
(Lorsque $V=\{x_1,\dots,x_n\}$, on \'ecrit plus simplement $\mathscr T_n$.)
Pour tout $\Q$-espace vectoriel ordonn\'e~$\Gamma$
et tout $a\in\Gamma^V$, un polyn\^ome tropical $t\in \mathscr T_V$
d\'efinit une application d'\'evaluation de $\Gamma^V$ dans~$\Gamma$.
On identifie deux polyn\^omes tropicaux \`a la classe de leurs \'evaluations;
ils forment alors un $\Q$-espace vectoriel.

\begin{defi}\label{defi.tl}
Soit $F$ un corps. Une structure de \emph{termes locaux} sur~$F$
est la donn\'ee, pour tout ensemble~$V$ de symboles de variables
et tout terme $t\in\mathscr T_V$, d'une application $R_t\colon (F^\times)^V\to\R$, v\'erifiant les conditions suivantes:
\begin{enumerate}[\upshape (i)]
\item Si deux polyn\^omes tropicaux $t, t'\in\mathscr T_V$ se d\'eduisent 
l'un de l'autre par une permutation~$\sigma$ de~$V$, 
il en est de m\^eme des applications~$R_t$ et~$R_{t'}$;
\item Si un symbole de variable de~$V$ n'appara\^{\i}t pas dans un terme $t\in\mathscr T_V$, l'application~$R_t$ ne d\'epend pas de la coordonn\'ee correspondante;
\item L'application $t\mapsto R_t$ est $\Q$-lin\'eaire;
\item Si $t\in\mathscr T_V$ et si $a\in (F^\times)^V$ v\'erifie
$t(v(a_1),\dots,v(a_n))\geq0$ pour tout $v\in\QVal(F)$,
alors $R_t(a_1,\dots,a_n)\geq 0$.
\end{enumerate}
\end{defi}

Une structure de termes locaux sur un corps~$F$ est dite \emph{globale}
si le terme $t=x$ en une variable~$x$ v\'erifie la relation
\begin{equation}
R_x = 0.
\end{equation}

\iffalse
\begin{rema}
La condition~(iv) dans la d\'efinition
d'une structure de termes locaux peut \^etre affaiblie
sans changer cette notion.
Consid\'erons en effet une formule~$\phi$ du langage des anneaux
dont les variables libres appartiennent \`a~$V$
et soit $t\in\mathscr T_V$ un polyn\^ome tropical.
Suppposons 
que pour tout corps~$F$, pour toute pseudo-valuation~$v$ sur~$F$
\`a valeurs dans un groupe ordonn\'e divisible~$\Gamma$,
et pour tout $a\in (F^\times)^V$ v\'erifiant $\phi(a)$, 
on a $t(v(a))\geq 0$ dans~$\Gamma$;
on dit alors que \emph{$\phi$ entra\^{\i}ne la positivit\'e de~$t$}.
Compte tenu du th\'eor\`eme de compacit\'e en th\'eorie des mod\`eles,
cela signifie que l'on peut d\'emontrer $t(v(a))\geq 0$
\`a partir des axiomes des pseudo-valuations sur un corps
et de la formule~$\phi(a)$.
D\'emontrons alors que pour tout $a\in (F^\times)^V$ 
tel que $t(v(a))\geq 0$ pour tout $v\in\QVal(F)$, on a 
$R_t(a)\geq 0$.

\footnote{\`A finir! Il semble qu'il faille d'abord \'etendre
une structure de GVF \`a une cl\^oture alg\'ebrique.}
\end{rema}
\fi

Nous avons vu qu'une structure de courbe ad\'elique sur un corps~$F$
donne lieu \`a une hauteur; 
elle fournit \'egalement des termes locaux: 
pour tout terme $t\in\mathscr T_V$
et tout $a\in (F^\times)^V$, on pose
\begin{equation}\label{eq.ca-tl}
 R_t(a) = \int_\Omega t((\log \abs{a_v}_\omega)_v)\, d\mu(\omega). 
\end{equation}
Lorsque $F$ v\'erifie la formule du produit, 
cette structure de termes locaux est globale.

Plus g\'en\'eralement, 
la formule
\begin{equation} \label{eq.tl-h}
h(a_1,\dots,a_n) =  R_t(a_1,\dots,a_n), \qquad t = \sup(x_1,\dots,x_n)
\end{equation}
montre comment une structure de termes locaux fournit naturellement 
une structure de hauteur, qui est globale, 
si celle de termes locaux l'est.

Pour construire, inversement, une structure de courbe ad\'elique
associ\'ee \`a une structure de termes locaux (ou \`a une structure de hauteur),
il faut, d'une part, identifier un espace mesur\'e sous-jacent naturel
et, d'autre part, identifier la forme lin\'eaire correspondant \`a l'int\'egrale.
Nous \'etudions d'abord cette seconde question.

\subsection{Fonctionnelles divisorielles}

Soit $F$ un corps. 

Nous allons d\'efinir maintenant un $\Q$-espace vectoriel ordonn\'e $\Div_\Q(F)$,
et m\^eme r\'eticul\'e, c'est-\`a-dire que deux \'el\'ements
quelconques poss\`edent une borne sup\'erieure et une borne inf\'erieure 
\parencite[chapitre~VI, \S1, \no9]{Bourbaki-2007l}, que nous appelerons le groupe des \emph{diviseurs} sur~$F$.
L'origine de cette terminologie appara\^{\i}tra au chapitre suivant
lorsque nous ferons le lien entre une variante relative de ce groupe
et la notion de diviseurs birationnels introduite par \textcite{Shokurov-1996}.
% Lorsque $F$ est une extension d'un corps~$k$,
% nous en d\'efinirons aussi une variante $\Div_\Q(F/k)$ qui
% s'identifie
% (proposition~\ref{prop.div-bdiv})
% au groupe des $\Q$-diviseurs birationnels de~\parencite{Shokurov-1996}
% sur un $k$-sch\'ema projectif int\`egre de corps des fonctions~$F$.
Ce groupe est en outre
muni d'un morphisme de groupes $\div\colon F^\times \to\Div_\Q(F)$
dont l'image engendre $\Div_\Q(F)$ comme $\Q$-espace vectoriel r\'eticul\'e.
Cette derni\`ere propri\'et\'e sera (peut-\^etre) famili\`ere aux lecteur\textperiodcentered ices 
au fait de la th\'eorie \'el\'ementaire des hauteurs, 
dans laquelle la construction des hauteurs na\"{\i}ves
associ\'ees \`a  un diviseur arbitraire
ne fait effectivement intervenir que des combinaisons lin\'eaires
et des bornes sup\'erieures.

La construction de $\Div_\Q(F)$ est formelle, et
d\'ebute par la consid\'eration du $\Q$-espace vectoriel
r\'eticul\'e universel $U_\Q(F)$ associ\'e au groupe~$F^\times$,
lui-m\^eme construit \`a partir du mono\"{\i}de r\'eticul\'e universel
(l'ensemble des parties finies de~$F^\times$,
l'addition correspondant au produit des parties
et l'op\'eration $\sup$ \`a la r\'eunion), 
puis adjonction de diff\'erences et tensorisation par~$\Q$.
Pour tout $a\in F^\times$, soit $\div(a)$ l'\'el\'ement de~$U_\Q(F)$
image de la partie $\{a\}$ de~$F^\times$ du mono\"{\i}de r\'eticul\'e universel;
cela fournit un morphisme de groupes $\div\colon F^\times\to U_\Q(F)$
dont l'image engendre~$U_\Q(F)$:
\`a multiplication par un nombre rationnel pr\`es, tout \'el\'ement~$\gamma$
de~$U_\Q(F)$ peut \^etre repr\'esent\'e sous la forme
\[ \gamma = \inf(\div(a_1),\dots,\div(a_m))- \inf(\div(b_1),\dots,\div(b_m)), \]
o\`u $a_1,\dots,a_m,b_1,\dots,b_m\in F^\times$.

Si $v$ est une valuation sur~$F$ dont le groupe de valeurs~$\Gamma$
est un $\Q$-espace vectoriel totalement ordonn\'e, il existe
une unique application $\Q$-lin\'eaire de~$U_\Q(F)$ dans~$\Gamma$,
encore not\'ee~$v$,
qui applique $\div(a)$ sur~$v(a)$ pour tout $a\in F^\times$
et le sup de deux \'el\'ements sur le sup de leurs images.
Lorsque $v$ est seulement une pseudo-valuation \`a valeurs
dans un $\Q$-espace vectoriel totalement ordonn\'e,
on peut d\'efinir
une application de $U_\Q(F)$ dans~$\overline\Gamma$:
l'\'el\'ement~$\gamma$ ci-dessus est appliqu\'e sur 
\[ v(\gamma) = \sup_i \inf_j v(a_i/b_j). \]
Pour $\gamma,\delta\in U_\Q(F)$,
on a alors $v(\sup(\gamma,\delta))=\sup(v(\gamma), v(\delta))$,
ainsi que $v(\gamma+\delta)=v(\gamma)+v(\delta)$ si le second terme est d\'efini.

Pour tout $\gamma\in U_\Q(F)$, l'application de $\PVal(F)$ dans~$\ovR$
donn\'ee par $v\mapsto v(\gamma)$ est continue et se prolonge par continuit\'e
\`a $\QVal(F)$.

\begin{defi}
On d\'efinit le groupe $\Div_\Q(F)$ des diviseurs de~$F$ 
comme le quotient de $U_\Q(F)$
par l'id\'eal des \'el\'ements~$\gamma$ de~$U_\Q(F)$ tels que $v(\gamma)=0$
pour tout $v\in\PVal(F)_*$.
\end{defi}

Pour $a\in F^\times$, l'image de $\div(a)$ est encore not\'ee $\div(a)$;
ces diviseurs sont dits \emph{principaux}.

Par construction, pour tout $v\in\PVal(F)_*$,
l'application $\gamma\mapsto v(\gamma)$ de~$U_\Q(F)$ dans~$\ovR$
d\'efinit une application additive de~$\Div_\Q(F)$ dans~$\ovR$.
Par passage \`a la limite, cela s'\'etend \`a tout~$v\in \QVal(F)$.
Pour tout $\gamma\in\Div_\Q(F)$, 
l'application de $\QVal(F)$ dans~$\ovR$ 
donn\'ee par $v\mapsto v(\gamma)$ est continue.

On dit qu'un diviseur~$\gamma$ est \emph{effectif} si $v(\gamma)\geq0$
pour tout $v\in\PVal(F)_*$ ;
par continuit\'e, cela vaut alors pour tout $v\in\QVal(F)$. 
Les diviseurs effectifs forment un c\^one $\Div_\Q^+(F)$ de~$\Div_\Q(F)$.

Pour tout nombre r\'eel~$c>0$, soit $\Gamma_c$ l'ensemble des
\'el\'ements de~$U_\Q(F)$ de la forme
\[ \inf_{1\leq i \leq m} \div (\sum_{j=1}^n c_{ij} b_{j})
 - \inf_{1\leq j \leq n} \div (b_j), \]
pour $m,n\geq 1$, $b_1,\dots,b_n\in F^\times$, 
et $(c_{ij})$ est une famille d'entiers relatifs tels que 
$\abs{\sum_{j=1}^n c_{ij}} <2^c$ pour tout $i\in\{1,\dots,m\}$.
\textcite[corollaire~5.17]{BenYaacovDesticHrushovskiEtAl-2024}
d\'emontrent la caract\'erisation suivante des \'el\'ements 
de~$U_\Q(F)$ dont l'image dans~$\Div_\Q(F)$ est effective.
\begin{prop}\label{prop.UQF-pos}
Soit $\alpha\in U_\Q(F)$.
Les conditions suivantes sont \'equivalentes:
\begin{enumerate}[\upshape (i)]
% \item Pour toute pseudo-valuation~$v$ sur~$F$, on a $v(\alpha)\geq0$;
%% \item Pour toute valuation~$v$ sur~$F$, on a $v(\alpha)\geq0$;
\item Pour tout $v\in \PVal(F)_*$, on a $v(\alpha)\geq 0$;
\item Pour tout $v\in \QVal(F)$, on a $v(\alpha)\geq 0$;
\item Pour tout nombre r\'eel~$c>0$, il existe un entier~$p\geq 1$
tel que $p\alpha\in \Gamma_{pc}$
\end{enumerate}
\end{prop}
\begin{proof}
% L'implication (i)$\Rightarrow$(ii) est \'evidente
L'\'equivalence (i)$\Leftrightarrow$(ii) d\'ecoule de la densit\'e
de~$\PVal(F)$ dans~$\QVal(F)$.

\smallskip
D\'emontrons l'implication~(iii)$\Rightarrow$(i). Soit $c>0$.
Soit $p$ un entier~$\geq 1$,
soit $b_1,\dots,b_n$ des \'el\'ements de~$F^\times$, 
et soit $(c_{ij})$ des entiers relatifs v\'erifiant 
$\abs{\sum_{j=1}^n c_{ij}} <2^{pc}$ pour tout $i\in\{1,\dots,m\}$,
tels que l'on ait
\[ p\alpha = \inf_{1\leq i \leq m} \div (\sum_{j=1}^n c_{ij} b_{j})
 - \inf_{1\leq j \leq n} \div (b_j)\ ; \]
pour tout~$i$, posons $a_i=\sum_{j=1}^n c_{ij}b_j$.
Soit $v$ une pseudo-valuation sur~$F$. 
Pour tout~$i$, on a $v(a_i) \geq \inf_j v(b_j) - pc v(2)$.
par suite
\[ v(p\alpha) = \inf_i \sup_j v(a_i/b_j) \geq - pc v(2). \]
Puisque $v(p\alpha)=p\alpha$, on en d\'eduit $v(\alpha)\geq -cv(2)$,
d'o\`u $v(\alpha)\geq0$ puisque $c$ est arbitraire.

\smallskip

V\'erifions l'implication (ii)$\Rightarrow$(iii).
On se ram\`ene au cas o\`u il existe $a_1,\dots,a_n\in F^\times$
tel que
$\alpha=-\inf(\div(a_1),\dots,\div(a_n))$.
Munissons l'anneau $\Z[T_1,\dots,T_n]$ de la norme donn\'ee par
\[ \norm f = \sum_k \abs{f_k} 2^{-\abs k c} \qquad (f = \sum f_k T^k). \]
D\'emontrons d'abord qu'il existe $f\in \Z[T]$ tel que $\norm f<1$
et $f(a)=1$; raisonnons par l'absurde.
Pour tout $x\in F$, soit $\norm x$ la borne inf\'erieure
des $\norm f$, pour $f\in \Z[T]$ tel que $f(a)=x$;
en particulier, $\norm x=+\infty$ si et seulement si $x$ n'est
pas de la forme~$f(a)$, pour $f\in\Z[T]$. 
Par hypoth\`ese, on a $\norm 1=1$;
les in\'egalit\'es $\norm{x+y}\leq \norm x+\norm y$ et $\norm {xy}\leq\norm x\,\norm y$ d\'ecoulent de la d\'efinition.
Ainsi, $\norm\cdot$ est une \og pseudo-seminorme\fg sous-multiplicative sur~$F$.
Le lemme de Zorn entra\^{\i}ne l'existence d'une pseudo-seminorme sous-multiplicative \emph{minimale}~$\abs\cdot$ sur~$F$ qui est major\'ee par~$F$.
On d\'emontre qu'une telle pseudo-seminorme est en fait multiplicative,
de sorte que $v=-\log\norm\cdot$ est une pseudo-valuation sur~$F$,
\'eventuellement triviale.
Puisque $\norm{T_i}=2^{-c}$, on a $\abs {a_i}\leq \norm{T_i}\leq 2^{-c}<1$
pour tout~$i$, c'est-\`a-dire $v(a_i)>c\log(2)>0$.
Ainsi, $v(\alpha)=-\inf(v(a_i))\leq - c\log(2)$.

Compte tenu de l'hypoth\`ese~(ii), la pseudo-valuation~$v$ est triviale.
Pour tout~$i$,
la minoration $v(a_i)>0$ entra\^{\i}ne que $v(a_i)=+\infty$ ;
en particulier $a_i$ appartient \`a l'anneau de valuation~$R$ de~$v$,
mais n'y est pas inversible.
Soit $\Delta$ le groupe ordonn\'e~$F^\times/R^\times$ (not\'e additivement)
et soit $w\colon F\to \overline{\Delta}$ 
l'application canonique ; c'est une pseudo-valuation.
Posons $\gamma=\inf(w(a_1),\dots,w(a_n))$;
on a $\gamma>0$.
Soit $\Delta_1$ le plus petit sous-groupe convexe de~$\Delta$
qui contient~$\gamma$ et soit $\Delta_0$
le plus grand sous-groupe convexe de~$\Delta_1$
qui ne contient pas~$\gamma$. On d\'eduit de~$w$
une pseudo-valuation~$w'$ \`a valeurs dans~$\Delta_1/\Delta_0$:
pour $x\in F$, on a
$w'(x)=+\infty$ si $w(x)>\Delta_1$,
$w'(x)=-\infty$ si $w(x)<\Delta_1$,
et $w'(x)$ est la classe de~$w(x)\in\Delta_1$ sinon.
Comme le groupe ordonn\'e~$\Delta_1/\Delta_0$ se plonge dans~$\R$, 
on d\'eduit de~$w'$ une pseudo-valuation \`a valeurs r\'eelles sur~$F$,
qui n'est pas triviale et telle que
$w'(a_i)\geq\gamma>0$ pour tout~$i$.  Par suite, 
$w'(\alpha)=-\inf(w'(a_i))\leq-\gamma<0$, ce qui contredit l'hypoth\`ese~(i).

Consid\'erons donc un polyn\^ome $f\in\Z[T_1,\dots,T_n]$
tel que $f(a_1,\dots,a_n)=1$ et $\norm f<1$.
Soit $P$ l'ensemble des polyn\^omes $g\in\Z[T_1,\dots,T_n]$
de degr\'e $\leq \deg(f)$ tel que $\norm g<1$ et $g(a)\neq0$.
On a donc 
\[ \alpha = - \sup_{1\leq i\leq n} \div(a_i)
 = \inf_{g\in P} \div(g(a)) 
  - \inf_{1\leq i\leq n} \inf_{g\in P} \div (a_i g(a)). \]
On d\'eduit alors de cette relation
que $\alpha\in\Gamma_{d}$, o\`u $ 2^d = (n+2^c+2)\deg(f)$.

Cette inclusion n'est cependant pas suffisante pour \'etablir~(iii).
On consid\`ere alors, pour tout entier~$m\geq 1$ l'ensemble~$H_m$
des mon\^omes de degr\'e~$m$ de $\Z[T_1,\dots,T_n]$; 
on a $\Card(H_m)=\binom{n+m-1}{n-1}$, et 
\[ m\alpha = - \inf _{h\in H_m} \div(h(a)). \]
En \'ecrivant $f^m(a)=0$, on obtient un polyn\^ome~$g$
de degr\'e~$m\deg(f)$ en des ind\'etermin\'ees~$U_h$, pour $h\in H_m$,
telles que $g( (h(a))_{h\in H_m})=1$ et $\norm g<1$.
L'argument pr\'ec\'edent d\'emontre que $m\alpha\in\Gamma_{d}$,
o\`u 
$ 2^{d} = (N+2^c+2) m \deg(f)$.
Lorsque $m$ tend vers l'infini, $d/m$ tend vers~$0$, ce qui conclut
la preuve de l'implication~(ii)$\Rightarrow$(iii).
\end{proof}

\begin{defi}
Une \emph{fonctionnelle divisorielle} sur~$F$ est une application $\Q$-lin\'eaire
de $\Div_\Q(F)$ dans~$\R$ qui est positive sur le c\^one $\Div_\Q^+(F)$.
\end{defi}

Une fonctionnelle divisorielle sur~$F$ est dite \emph{globale}
si elle s'annule sur les diviseurs principaux.

\begin{prop}\label{prop.fd-h}
Si $\lambda$ est une fonctionnelle divisorielle sur~$F$,
la formule
\[ h_n(x_1,\dots,x_n) = -\lambda (\inf(\div(x_1),\dots,\div(x_n))) \]
d\'efinit une hauteur sur~$F$.
Cette relation induit une bijection entre fonctionnelles
divisorielles et hauteurs, par laquelle
fonctionnelles globales et hauteurs globales se correspondent.
\end{prop}
\begin{proof}
Le seul point non \'evident est le comportement de~$h_n$ par rapport
\`a l'addition.
On remarque tout d'abord que pour $x,y\in F^\times$ tels que $x+y\neq0$,
le diviseur $\div(x+y)-\inf(\div(x), \div(y))+ \inf(\div(2),0)$
est effectif: en effet, pour toute pseudo-valuation $v\in\PVal(F)_*$,
et, par passage \`a la limite, pour toute $v\in\QVal(F)$, 
on a $v(x+y) \geq \inf(v(x),v(y))-\inf(v(2),0)$. 
Par suite, pour tous $x,y\in F^n$, on a
\begin{multline*}
 \inf(\div(x_1+y_1),\dots,\div(x_n+y_n)) \\
 \geq \inf(\div(x_1),\dots,\div(x_n),\div(y_1),\dots,\div(y_n))
 - \inf (v(2), 0). \end{multline*}
Puisque $\lambda$ est une fonctionnelle divisorielle, on en d\'eduit
\begin{align*}
 h(x+y) & = - \lambda (\inf(\div(x+y))) \\
& \leq -\lambda (\inf(\div(x),\div(y))) + \lambda(\inf(\div(2),0)) \\
& = h(x,y) + h(2,1). \end{align*}

Inversement, soit $h$ une hauteur sur~$F$.
Comme $\Div_\Q(F)$ est engendr\'e par les \'el\'ements
de la forme $\inf(\div(x_1),\dots,\div(x_n))$, 
on voit que~$h$ ne peut \^etre associ\'ee qu'\`a au plus
une fonctionnelle divisorielle. Pour d\'emontrer l'existence
d'une telle fonctionnelle,
on remarque qu'il existe
un unique morphisme de groupes~$\lambda$ de $U_\Q(F)$ dans~$\R$
tel que $\lambda(\inf(\div(x_1),\dots,\div(x_n)))=-h(x_1,\dots,x_n)$,
et il suffit de prouver que $\lambda$ est positive
en tout $\gamma\in U_\Q(F)$ tel que $v(\gamma)\geq0$ pour tout $v\in\QVal(F)$.
On utilise pour cela la caract\'erisation de la proposition~\ref{prop.UQF-pos}.
\end{proof}

\subsection{Mesures}\label{ss.mesures}

On a vu comment la notion de courbe ad\'elique de~\textcite{ChenMoriwaki-2020}
permet de d\'efinir des stuctures de hauteurs, de termes locaux, 
ou des fonctionnelles divisorielles. 
Inversement, pour montrer que ces structures
fournissent un objet analogue  \`a une courbe ad\'elique,
il faut leur associer un espace \og canonique\fg et, 
pour autant que l'on souhaite
construire une mesure via le th\'eor\`eme de Riesz, 
assurer qu'il soit localement compact et 
munir l'espace de ses fonctions continues \`a support compact
d'une forme lin\'eaire positive.

Compte tenu de la d\'efinition d'une fonctionnelle divisorielle
associ\'ee \`a une courbe ad\'elique, l'espace~$\Val(F)_*$
des valuations r\'eelles non triviales sur~$F$ semble un candidat naturel
% mais comme une telle valuation est \'equivalente 
% \`a chacun de ses multiples par un nombre r\'eel strictement positif,
mais il ne convient pas vraiment,
pas plus que ses variantes $\PVal(F)_*$ ou $ \QVal(F)$:
il faudrait plut\^ot consid\'erer leurs quotients par la relation
d'\'equivalence induite par l'action de~$\R_+^*$ qui identifie une valuation
\`a ses multiples par un nombre r\'eel strictement positif.
Comme cette action n'est en g\'en\'eral pas propre,
l'espace quotient $\PVal(F)_*/\R_+^*$ n'est pas s\'epar\'e
\parencite[remarque~2.12]{BenYaacovDesticHrushovskiEtAl-2024}.

D'autre part, si les \'el\'ements de~$\Div_\Q(F)$ donnent lieu
\`a des fonctions continues sur~$\QVal(F)$, ces fonctions
sont \emph{homog\`enes} de degr\'e~$1$ sous l'action de~$\R_+^*$.

Pour une repr\'esentation \og mesur\'ee\fg des structures de hauteur ou
de fonctionnelles divisorielles,
on est donc amen\'e 
\`a introduire des conditions explicites de normalisation
sur l'espace~$\QVal(F)$.

Pour tout \'el\'ement positif de~$\alpha\in U_\Q(F)$,
identifi\'e \`a son image dans~$\Div^+_\Q(F)$, 
on note $\Omega^1_\alpha$ le sous-espace
de $\QVal(F)$ form\'e des~$v$ tels que $v(\alpha)=1$;
par exemple, si $\alpha=\inf(\div(a_1),\dots,\div(a_n))$,
cette condition s'\'ecrit $\inf(v(a_1),\dots,v(a_n))=1$.
C'est un espace \emph{compact} : c'est la trace,
sur le sous-espace ferm\'e de l'espace compact~$\ovR^F$
d\'efini par la condition~$v(\alpha)=1$, 
de l'adh\'erence de $\PVal(F)_*$ dans~$\ovR^F$.
L'application de passage au quotient
de~$\QVal(F)$ sur~$\QVal(F)/\R_+^*$
induit un hom\'eomorphisme de~$\Omega^1_\alpha$ sur une
partie compacte (s\'epar\'ee !)~$\Omega_\alpha$ de~$\QVal(F)/\R_+^*$.
Lorsque~$\alpha$ parcourt~$U_\Q(F)$, ces
parties compactes recouvrent~$\QVal(F)/\R_+^*$.

Soit $\mathscr D_\alpha$ l'ensemble des $\gamma\in\Div_\Q(F)$
pour lesquels il existe $n\in\N$ tel que $-n\cdot\alpha\leq \gamma\leq n\cdot\alpha$.
C'est un sous-espace vectoriel r\'eticul\'e de $\Div_\Q(F)$.
Pour tout $\gamma\in\mathscr D_\alpha$, la fonction~$j(\gamma)$
sur~$\Omega^1_\alpha$ donn\'ee par $j(\gamma)(v)=v(\gamma)$
est continue et \`a valeurs r\'eelles. 
Ces fonctions~$j(\gamma)$ s\'eparent les points de~$\Omega^1_\alpha$
et contienent les fonctions constantes (prendre $\gamma=\alpha$!).
% 
% N'importe quoi :
% --------------
% Soit $v,w$ deux points distincts de~$\Omega^1_\alpha$ ;
% il existe donc $f\in F^\times$ tel que $v(f)\neq w(f)$.
% Par ailleurs, si $\gamma$ et $\gamma'$ sont deux \'el\'ements
% distincts de~$\mathscr D_\alpha$, la construction de $\Div_\Q(F)$
% entra\^{\i}ne qu'il existe une pseudo-valuation $v\in\QVal(F)$
% telle que $v(\gamma)\neq v(\gamma')$. Cette pseudo-valuation
% v\'erifie n\'ecessairement $v(\alpha)\neq+\infty$\footnote{Je ne comprends
% pas pourquoi\dots}
% et $0<v(\alpha)$ (on aurait, sinon, $v(\gamma)=v(\gamma')=0$);
% ainsi, la pseudo-valuation~$v(\alpha)^{-1} \cdot v$
% distingue~$\gamma$ et~$\gamma'$.
La version r\'eticul\'ee du th\'eor\`eme de Stone 
\parencite[chapitre~X, \S4, \no2, p.~35, corollaire]{Bourbaki-2007m},
entra\^{\i}ne que l'image de~$j$ est dense dans l'espace des fonctions
continues sur~$\Omega^1_\alpha$ (pour la topologie de la convergence compacte).
Il existe donc une unique  mesure $\mu^1_\alpha$ sur~$\Omega^1_\alpha$
pour laquelle
\[ \lambda(\gamma) = \int_{\Omega^1_\alpha} v(\gamma)\,d\mu^1_\alpha(v) \]
pour tout $\gamma\in\mathscr D_\alpha$,
et une unique mesure $\mu_\alpha$ sur~$\Omega_\alpha$
d\'eduite de~$\mu^1_\alpha$ par l'hom\'eomorphisme
de~$\Omega^1_\alpha$ sur~$\Omega_\alpha$.

Les mesures~$\mu_\alpha$
v\'erifient les relations de compatibilit\'e suivantes.
Pour $\alpha,\beta\in U_\Q(F)^+$,
la condition $v(\beta)=+\infty$ 
d\'efinit un ensemble $\mu_\alpha$-n\'egligeable dans~$\Omega_\alpha$.
Soit alors $f_\alpha^\beta$ l'application
sur $\Omega_\alpha\cap\Omega_\beta$ qui, pour tout $v\in\QVal(F)$
tel que $0<v(\alpha), v(\beta)<+\infty$,
applique la classe de~$v$ sur le quotient~$v(\alpha)/v(\beta)$,
et qui applique toute autre classe en~$0$.
Avec ces notations, les restrictions \`a~$\Omega_\alpha\cap\Omega_\beta$
des mesures~$\mu_\alpha$ et $\mu_\beta$ v\'erifient la relation
$\mu_\alpha = f_{\alpha}^\beta \cdot \mu_\beta $.

\begin{defi}
Une \emph{structure locale de mesures} sur~$F$
est la donn\'ee, pour tout $a\in F^\times$,
d'une mesure~$\mu_a$ sur~$\Omega_{\sup(0,\div(a))}$,
de telle fa\c{c}on que les conditions suivantes soient satisfaites,
pour tous $a,b\in F^\times$:
\begin{enumerate}[\upshape (i)]
\item L'ensemble des $v\in\Omega_a$ tels que $v(b)=+\infty$
est $\mu_a$-n\'egligeable;
\item Sur $\Omega_a\cap\Omega_b$, on a $\mu_a = f_a^b\cdot \mu_b$,
o\`u $f_a^b$ est la fonction sur~$\Omega_a\cap\Omega_b$
(continue en dehors d'un ensemble $\mu_a$ et $\mu_b$-n\'egligeable)
qui applique la classe de~$v$ sur $v(a)/v(b)$.
\end{enumerate}
\end{defi}

Une structure locale de mesures $(\mu_a)$ sur~$F$ est dite \emph{globale}
si l'on a
\begin{equation}
\mu_{a^{-1}}(\Omega_{a^{-1}})=\mu_a(\Omega_a)
\end{equation}
pour tout $a\in F^\times$.

\begin{prop}
Soit $\lambda$ une fonctionnelle divisorielle sur~$F$.
Pour tout $a\in F^\times$, notons $\mu_a$ la mesure
associ\'ee \`a l'\'el\'ement positif $\sup(0,\div(a))$ de~$\Div_\Q(F)$.
La famille $(\mu_a)_{a\in F^\times}$ est une structure locale de mesures 
sur~$F$;
elle est globale si et seulement si $\lambda$ l'est.
\end{prop}

\subsection{\'Equivalences}\label{ss.gvf}

Le th\'eor\`eme principal de cette premi\`ere partie
est que tous ces regards sur la notion de corps globalement valu\'e
sont \'equivalents.

\begin{theo}
Soit $F$ un corps. Les constructions pr\'ec\'edentes
fournissent des bijections entre structures de hauteur,
structures de termes locaux, fonctionnelles divisorielles
et structures locales de mesures sur~$F$, 
qui induisent des bijections entre leurs versions globales.
\end{theo}
\begin{proof}
On a d\'efini des applications :
\begin{itemize}
\item Des structures de termes locaux vers les structures de hauteur
(formule~\eqref{eq.tl-h}) ;
\item Des fonctionnelles divisorielles vers les structures
de hauteur, et inversement (proposition~\ref{prop.fd-h}) ;
\item Des fonctionnelles divisorielles vers les structures locales
de mesures (\S\ref{ss.mesures}).
\end{itemize}
La construction d'une structure de termes locaux associ\'ee
\`a une structure locale de mesures passe par la construction
de mesures~$\mu$ sur~$\QVal(F)$ associ\'ees \`a une structure locale~$(\mu_a)$
de mesures sur~$F$.
Ces mesures~$\mu$ ont la propri\'et\'e que pour tout~$a\in F^\times$,
la restriction de~$\mu$
\`a l'ouvert des $v\in\QVal(F)$ tels que $0<v(a)<+\infty$
est de la forme $(s_a)_* \mu_a$, o\`u $s_a\colon \Omega_a \to\QVal(F)$
est une section bor\'elienne de l'application de passage au quotient.
Pour tout polyn\^ome tropical~$t\in\mathscr T_V$ et tout $a\in (F^\times)^V$,
on d\'emontre que la formule
(comparer avec la relation~\eqref{eq.ca-tl})
\[ R_t(a) = \int_{\QVal(F)} t(v(a))\, d\mu(v), \]
d\'efinit une structure de termes locaux sur~$F$.
De plus, on d\'emontre que cette structure ne d\'epend pas
de la mesure~$\mu$ choisie, mais seulement de la famille~$(\mu_a)$.
\end{proof}

\subsection{Corps globalement valu\'es}

Un \emph{corps multiplement valu\'e} est un corps muni de l'une
des structures pr\'ec\'edentes. On dit que c'est un \emph{corps globalement valu\'e}
si cette structure est globale.
Il est dit \emph{purement non archim\'edien} si 
l'on peut prendre $e=0$ dans sa structure de hauteur.

\begin{exem}[Restriction et extension]
Soit $F\to K$ une extension de corps.  

Une structure
de corps globalement valu\'e sur~$K$
induit une telle structure sur~$F$: il suffit de restreindre
la structure de hauteur, ou la structure de termes locaux.

Inversement, supposons $F$ munie d'une structure de corps globalement
valu\'e. Si l'extension~$K$ de~$F$ est quasi-galoisienne
\parencite[chapitre~V, \S9, \no3; certains auteurs disent \og normale \fg]{Bourbaki-2007l},
il existe une unique structure de corps globalement valu\'e
sur~$K$ qui induise celle de~$F$ et soit invariante par $\Aut_F(K)$.

Il s'ensuit qu'une de corps globalement valu\'e sur~$F$
s'\'etend de fa\c{c}on canonique \`a toute extension alg\'ebrique;
on parlera d'extension \emph{sym\'etrique}.
Dans le cas du corps~$\overline\Q$ des nombres alg\'ebriques
et d'une cl\^oture alg\'ebrique du corps~$k(T)$,
on qualifiera de \emph{standard}
la structure de corps globalement valu\'e  qui prolonge
celle de~$\Q$ ou de~$k(T)$.
\end{exem}

\begin{exem}[Corps globaux] \label{exem-aw}
\textcite[th\'eor\`eme 3]{ArtinWhaples-1945} ont caract\'eris\'e les
\emph{corps globaux}~$F$ de la th\'eorie du corps de classes.
Ils partent de corps munis d'un ensemble~$M_F$ de valeurs absolues
non triviales sur~$F$, deux \`a deux non \'equivalentes,
donnant lieu \`a une formule du produit 
\[ \prod_{v\in M_F} \abs a_v = 1 \]
(o\`u, pour tout $a\in F^\times$, presque tous les facteurs sont \'egaux \`a~$1$),
et d\'emontrent que les corps globaux sont alors caract\'eris\'es par
l'existence d'une valeur absolue $v\in M_F$
qui, ou bien soit archim\'edienne,
ou bien soit discr\`ete \`a corps r\'esiduel fini.

Ils fournissent en particulier des exemples de corps globalement valu\'es.

Dans le premier cas, $F$ est un corps de nombres;
on peut d\'ecrire l'espace des fonctionnelles divisorielles globales
et en d\'eduire  que
\emph{toute structure de corps globalement
valu\'e sur~$F$ est multiple de la structure standard.}

Dans le second cas, $F$ est le corps des fonctions d'une
courbe propre, r\'eguli\`ere, g\'eom\'etriquement int\`egre~$C$ sur un corps fini~$k$.
De m\^eme, on peut d\'ecrire l'espace des fonctionnelles divisorielles
globales qui sont triviales sur~$k$ et en d\'eduire que 
\emph{toute structure de corps globalement
valu\'e sur~$F$ qui est triviale sur~$k$ est multiple de la structure
standard.}

Cette propri\'et\'e s'\'etend \`a leurs extensions alg\'ebriques.
\end{exem}

\begin{exem}
Soit $k$ un corps alg\'ebriquement clos et soit $F$ le
corps des fonctions d'un $k$-sch\'ema propre~$X$, suppos\'e irr\'eductible et normal.
Soit $C\to S$ une famille plate de courbes (propres, irr\'eductibles) sur~$X$ 
param\'etr\'ee par un sch\'ema non vide connexe~$S$;
on suppose que cette famille n'est contenue 
dans aucun sous-sch\'ema strict de~$X$.
Soit $\Omega$ l'ensemble des hypersurfaces int\`egres de~$X$.
On le munit de la mesure discr\`ete~$\mu_C$ pour laquelle
la masse~$\mu_C(H)$ d'un point $H\in\Omega$
est le degr\'e d'intersection commun $C_s\cdot H$,
pour $s\in S$ tel que $C_s\not\subset H$.
Comme $X$ est normale, toute hypersurface $H\in\Omega$ 
est le centre d'une unique valuation discr\`ete, normalis\'ee;
notons-la $v_H$.
On a ainsi muni~$F$ d'une structure de courbe ad\'elique.
Pour toute fonction~$f\in k(X)^\times$, il existe $s\in S$
tel que la courbe~$C_s$ ne soit contenue dans aucune composante
de~$\div(f)$; alors, 
\[ \sum_{H\in\Omega} v_H (f) \mu_C(H) = C_s \cdot \div(f) = 0. \]
On obtient ainsi une structure de corps globalement valu\'e sur~$F$;
sa restriction \`a~$k$ est la structure triviale.
\end{exem}

Les exemples pr\'ec\'edents peuvent \^etre d\'ecrits via
la notion de courbe ad\'elique, mais 
les prochains sont construits par ultraproduits.

\begin{exem}
Soit $(F_i)_{i\in I}$ une famille de corps globalement valu\'es;
on suppose qu'il existe un m\^eme nombre r\'eel~$e$
pour lequel leurs fonctions hauteurs (que nous noterons indistinctement~$h$)
satisfont l'in\'egalit\'e~(vi) de la d\'efinition~\ref{defi.h}.
Soit $\mathfrak u$ un ultrafiltre (non principal) sur~$I$.
Consid\'erons la partie~$B$ de $\prod_{i\in I} F_i$
form\'ee des familles $(a_i)$ telles que la famille $(h_2(1,a_i))$ soit 
\'eventuellement born\'ee suivant l'ultrafiltre~$\mathfrak u$.
C'est un sous-anneau.
Munissons-le de la relation d'\'equivalence d\'efinie par
$(a_i)\sim (b_i)$ si et seulement si l'ensemble des $i\in I$
tels que $a_i=b_i$ appartient \`a~$\mathfrak u$.
Cette relation est compatible avec la structure d'anneau de~$B$,
de sorte que l'ensemble quotient~$F_{\mathfrak u}$ 
h\'erite d'une structure d'anneau;
c'est en fait un corps.

Pour toute suite finie $(a_1,\dots,a_n)$ d'\'el\'ements de~$B$,
la famille de nombres r\'eels $(h_n(a_{1,i},\dots,a_{n,i}))_i$ est born\'ee,
donc poss\`ede une limite suivant l'ultrafiltre~$\mathfrak u$.
Cette limite ne d\'epend que des images~$\alpha_1,\dots,\alpha_n$
dans~$F_{\mathfrak u}$ de $a_1,\dots,a_n$; 
notons-la $h_n(\alpha_1,\dots,\alpha_n)$.
On obtient ainsi une structure de hauteur sur~$F_{\mathfrak u}$.

On peut appliquer cette construction \`a une suite
de structures de corps globalement valu\'es sur $\overline\Q$
dont les erreurs archim\'ediennes tendent vers~$0$.
Cela fournit un corps globalement valu\'e, alg\'ebriquement
clos de caract\'eristique z\'ero, induisant la structure
triviale sur~$\overline\Q$, et donc non archim\'edien.

On peut \'egalement l'appliquer \`a une suite de 
structures de corps globalement valu\'ees sur des cl\^otures
alg\'ebriques $K_p$ de $\F_p(T)$, o\`u $I$ est l'ensemble des nombres premiers.
On en d\'eduit un corps globalement valu\'e, alg\'ebriquement clos, 
de caract\'eristique z\'ero, induisant \'egalement la structure
triviale sur $\overline\Q$.
\end{exem}

\begin{exem}
Soit $\mathscr M$ le corps des fonctions m\'eromorphes sur~$\C$.
Pour $f\in\mathscr M^\times$ et $a\in\C$, notons $v_a(f)$
l'ordre d'annulation de~$f$ en~$a$.
Soit $f\in\mathscr M^\times$;
supposons $f(0)\neq 0$. Pour tout nombre r\'eel~$r>0$,
 la formule de Jensen s'\'ecrit
\[ \log \abs{f(0)}^{-1} = \sum_{0<\abs a <r} v_a(f) \log \frac r{\abs a}
 + \int_{0}^{2\pi} \log \abs{f(re^{i\theta})}^{-1}\frac{d\theta}{2\pi}. \]
La premi\`ere partie du membre de droite de cette
formule se d\'eduit de valuations sur~$\mathscr M$;
le second terme est de nature diff\'erente, si ce n'est
que pour tout $z\in\C$, l'application $f\mapsto \log \abs{f(z)}^{-1}$
est une pseudo-valuation sur~$\mathscr M$.

Pour $a\in\C$, notons~$\delta_a$ la mesure de Dirac en~$a$;
pour $r\in\R_+^*$, notons~$\gamma_r$ la mesure de Haar normalis\'ee
sur le cercle de centre~$0$ et de rayon~$r$;
soit alors $\mu_r$ la mesure positive sur~$\C$ donn\'ee par
\[ \mu_r = \sum_{0<\abs a<r} \log \frac r{\abs a} \delta_a
 + \gamma_r . \]

Soit $\eta\colon\R_+^*\to\R_+^*$ une fonction tendant vers~$0$
et soit $\mathfrak u$ un ultrafiltre non principal sur~$\R_+^*$.
Pour $f\in\mathscr M$ et $r>0$, posons
\[ \Ht_r (f) = \mu_r (\sup(0, \log \abs f^{-1}) ).\]
En fait, $\Ht_r(f)$ n'est autre que la fonction
caract\'eristique~$T_f(r)$ attach\'ee \`a~$f$ par la th\'eorie de Nevanlinna,
voir \parencite{Vojta-1987}.

Soit alors $\mathscr M_{\eta,\mathfrak u}$ 
l'ensemble des fonctions $f\in\mathscr M$
telles que la famille $(\eta(r) \Ht_r(f))_r$  soit born\'ee selon~$\mathfrak u$;
notons alors $\Ht(f)$ la limite de cette famille selon~$\mathfrak u$.
Alors, $\mathscr M_{\eta,\mathfrak u}$
est un sous-corps de~$\mathscr M$.
Pour toute suite finie $(f_1,\dots,f_n)$ dans~$\mathscr M_{\eta,\mathfrak u}$,
la famille de nombres r\'eels
\[ ( \eta(r) \mu_r(\sup(\log\abs{f_1}^{-1},\dots,\log\abs{f_n}^{-1})) \]
est born\'ee selon~$\mathfrak u$; notons $h_n(f_1,\dots,f_n)$ sa limite.
Cela d\'efinit une structure de corps globalement valu\'e
sur le sous-corps~$\mathscr M_{\eta,\mathfrak u}$.
qui est triviale sur le sous-corps des fonctions constantes.
\end{exem}

\section{Th\'eorie de l'intersection}

Dans ce chapitre, on s'int\'eresse plus particuli\`erement 
aux structures de corps globalement valu\'es
sur un corps~$F$ qui v\'erifient l'une des deux hypoth\`eses suivantes:
\begin{itemize} 
\item Le corps~$F$ est de type fini sur un sous-corps~$k$,
et sa structure de corps globalement  valu\'e induit
la structure triviale sur~$k$ (cas g\'eom\'etrique);
\item Le corps~$F$ est de type fini sur~$\Q$
et sa structure de corps globalement valu\'e induit
la structure standard sur~$\Q$ (cas arithm\'etique).
\end{itemize}
Nous allons voir comment la th\'eorie de l'intersection
(g\'eom\'etrique ou arithm\'etique) permet de d\'ecrire
essentiellement toutes ces structures de corps globalement valu\'e.

\subsection{Densit\'e des valuations}

Commen\c{c}ons par introduire une variante \og relative\fg 
des corps globalement valu\'es o\`u l'on impose la structure triviale
sur un sous-corps.

Soit $F$ un corps  et $k$ un sous-corps de~$F$.
On note $\Val(F/k)$ et $\PVal(F/k)$ les sous-espaces de~$\Val(F)$ et~$\PVal(F)$
form\'es des valuations et pseudo-valuations sur~$F$ 
dont la restriction \`a~$k$ est triviale.
On note encore $\Val(F/k)_*$ et $\PVal(F/k)_*$ les sous-espaces 
constitu\'es des valuations et pseudo-valuations non triviales,
ainsi que $\QVal(F/k)$ l'adh\'erence de~$\PVal(F/k)_*$ dans~$\QVal(F)$.

On note alors plut\^ot $U_\Q(F/k)$ le groupe r\'eticul\'e $U_\Q(F)$,
et l'on d\'efinit $\Div_\Q(F/k)$ comme son quotient 
par l'id\'eal des $\gamma\in U_\Q(F/k)$
tels que $v(\gamma)=0$ pour tout $v\in\QVal(F/k)$.
Ce groupe est muni d'une relation d'ordre analogue
\`a celle de~$\Div_\Q(F)$
et on note $\Div_\Q^+(F/k)$ le c\^one effectif correspondant,
image des~$\gamma\in U_\Q(F/k)$ tels que $v(\gamma)\geq0$ pour
tout $v\in\QVal(F/k)$.

\begin{prop}
Soit $k$ un corps et soit~$F$ une extension de type fini de~$k$.
L'espace $\Val(F/k)$ des valuations (\`a valeurs r\'eelles) sur~$F$
qui sont triviales sur~$k$ est dense dans l'espace~$\PVal(F/k)$
des pseudo-valuations sur~$F$ qui sont triviales sur~$k$.
\end{prop}
\begin{proof}
Munissons~$k$ de la valeur absolue triviale.

Si $X$ est un $k$-sch\'ema de type fini,
l'espace $k$-analytique~$X^\an$ au sens de~\textcite{Berkovich-1990} 
a pour ensemble sous-jacent l'ensemble des couples $(x,p)$,
o\`u $x\in X$ et $p$ est une valeur absolue sur le
corps r\'esiduel $k(x)$ qui induit la valeur absolue triviale sur~$k$.
Soit $\pi\colon X^\an\to X$ l'application donn\'ee par $\pi(x,p)=x$.
Cet ensemble est muni de la topologie la moins fine telle que
pour tout ouvert~$U$ de~$X$
et toute fonction $f\in\mathscr O_X(U)$,
l'ensemble $\pi^{-1}(U)$ soit ouvert dans~$X^\an$ et
l'application $(x,p)\mapsto p(f(x))$ de~$\pi^{-1}(U)$ dans~$\R$
soit continue.

L'application~$\pi$ est continue et surjective.

Si $X$ est s\'epar\'e, l'espace~$X^\an$ est localement compact;
si $X$ est propre, il est compact.

Supposons~$X$ int\`egre et soit~$\eta_X$ son point g\'en\'erique.
La partie~$\pi^{-1}(\eta_X)$ de~$X^\an$ s'identifie \`a l'espace~$\Val(F/k)$
des valeurs absolues sur~$k(X)$ qui sont triviales sur~$k$;
cette partie est dense dans~$X^\an$ (\cite{Berkovich-1990}, \S3.5).
% \footnote{Peut-\^etre faudrait-il prouver ce dernier fait\dots}

Plus g\'en\'eralement, toute pseudo-valuation
\`a valeurs r\'eelles sur~$F$ qui est triviale sur~$k$
est une valuation sur un anneau de valuation de~$F$ contenant~$k$.
On d\'efinit ainsi une application continue $\PVal(F/k)\to X^\an$.

Cette construction est fonctorielle. Soit $u\colon X\to Y$ un morphisme
de $k$-sch\'emas, soit $x\in X$ ; le morphisme~$u$ 
induit un morphisme de corps $u^*\colon k(u(x))\to k(x)$
de sorte qu'une valeur absolue~$p$ sur~$k(x)$ qui 
est triviale sur~$k$ induit, via~$u^*$,
une valeur absolue similaire~$u^*p$ sur~$k(u(x))$.
L'application $u^\an\colon X^\an\to Y^\an$
qui applique $(x,p)$ sur~$ (u(x),u^*p)$ est continue.
Si $u$ est surjective, il en est de m\^eme de~$u^\an$.

Restreignons-nous maintenant au cas o\`u les $k$-sch\'emas
consid\'er\'es sont des \emph{mod\`eles} de~$F$, c'est-\`a-dire
des $k$-sch\'emas propres et int\`egres munis d'un isomorphisme $F\simeq k(X)$,
et les morphismes $u\colon X\to Y$ sont compatibles \`a ces isomorphismes
(en particulier birationnels).

On en d\'eduit une application continue 
\[ \PVal(F/k) \to \varprojlim_{F \simeq k(X)}  X^\an. \]
On prouve que cette application est un \emph{hom\'eomorphisme}:
l'argument est similaire \`a l'identification, classique
dans le contexte des espaces de Riemann de Zariski,
de la  limite projective du syst\`eme
des sch\'emas birationnels \`a un sch\'ema int\`egre donn\'e
comme un espace de valuations.
La densit\'e de $\Val(F/k)$ dans chaque~$X^\an$ entra\^{\i}ne la densit\'e requise.
\end{proof}

\begin{coro}
Soit $k$ un corps et soit $F$ une extension de type fini de~$k$.
Pour qu'un \'el\'ement $\gamma\in U_\Q(F/k)$ soit positif,
il faut et il suffit que $v(\gamma)\geq0$ pour toute
valuation r\'eelle~$v$ sur~$F$.
\end{coro}

\begin{prop}
Soit $F$ un corps d\'enombrable. L'espace $\Val(F)$
des valuations (\`a valeurs r\'eelles) sur~$F$
est dense dans l'espace~$\PVal(F)$.
\end{prop}
\begin{proof}
La d\'emonstration est analogue mais utilise 
la possibilit\'e, au sein de la th\'eorie de~\textcite{Berkovich-1990},
d'associer un espace analytique~$X^\an$ \`a un $\Z$-sch\'ema de type fini~$X$,
ainsi que des r\'esultats de~\textcite{LemanissierPoineau-2024}
sur ces espaces analytiques.
\end{proof}

\begin{coro}
Soit $F$ un corps d\'enombrable.
Pour qu'un \'el\'ement $\gamma\in U_\Q(F)$ soit positif,
il faut et il suffit que $v(\gamma)\geq0$ pour toute
valuation r\'eelle~$v$ sur~$F$.
\end{coro}

\subsection{Diviseurs birationnels et fonctionnelles divisorielles}

Soit $k$ un corps et soit $X$ un $k$-sch\'ema propre et int\`egre.
Pour simplifier la discussion, on suppose que $X$ est normal,
de sorte qu'un diviseur de Cartier sur~$X$ n'est autre qu'un diviseur
sur~$X$ qui est localement repr\'esent\'e comme le diviseur 
d'une fonction rationnelle (non nulle).
Les diviseurs birationnels (ou b-diviseurs) 
introduits par~\textcite{Shokurov-1996}
sont les diviseurs de Cartier sur un mod\`ele birationnel propre normal 
arbitraire~$X'$ de~$X$, modulo l'identification de~$D$ et~$\pi^*(D)$
si $\pi \colon X''\to X'$ est un $X$-morphisme propre birationnel
et $D$ est un diviseur de Cartier sur~$X'$.
Plus synth\'etiquement,  on a
\[ \bDiv(X) = \varinjlim_{X'} \Div(X'). \]
On peut m\^eme se restreindre au syst\`eme cofinal des mod\`eles projectifs.

Le groupe ab\'elien~$\bDiv(X)$ est muni d'une relation d'ordre pour laquelle
les \'elements positifs
sont ceux qui sont repr\'esent\'es par un diviseur de Cartier effectif 
sur un mod\`ele ad\'equat.
Ce groupe ab\'elien ordonn\'e est r\'eticul\'e:
lorsque deux diviseurs birationnels sont repr\'esent\'es sur un m\^eme mod\`ele
par des diviseurs de Cartier dont les composantes irr\'eductibles
sont deux \`a deux disjointes ou \'egales, leur borne
inf\'erieure est d\'efinie en prenant, pour chacune d'entre elle,
la borne inf\'erieure des coefficients ; on se ram\`ene \`a ce cas
par \'eclatement.  
L'exemple le plus simple \'eclaire cette d\'efinition: 
consid\'erons dans $\bDiv(\P_2)$, les deux diviseurs de Cartier~$D_1$ et~$D_2$ 
correspondant aux deux axes de coordonn\'ees $x_1=0$ et $x_2=0$
dans~$\P_2$; alors $\inf(D_1,D_2)$ est d\'efini par le diviseur exceptionnel 
sur l'\'eclatement de leur point d'intersection.

Dans la suite, on utilisera plut\^ot le $\Q$-espace vectoriel
r\'eticul\'e $\bDiv(X)_\Q$ des $\Q$-diviseurs birationnels sur~$X$.

\begin{prop}\label{prop.div-bdiv}
Soit $k$ un corps et soit $X$ un $k$-sch\'ema propre, int\`egre et normal.
Il existe un unique morphisme d'espaces vectoriels r\'eticul\'es
\[ \Div_\Q(k(X)/k) \to \bDiv(X)_\Q \]
qui, pour $f\in k(X)^\times$, applique $\div(f)$ sur le diviseur  birationnel
sur~$X$ correspondant au diviseur de Cartier principal $\div(f)$ sur~$X$.
C'est un isomorphisme.
\end{prop}
\begin{proof}
Soit $\theta\colon U_\Q(k(X)/k)\to\bDiv(X)_\Q$ l'unique
morphisme d'espaces vectoriels r\'eticul\'es qui,
pour tout $f\in k(X)^\times$, applique~$\div(f)$ sur le diviseur birationnel
sur~$X$ correspondant au diviseur de Cartier principal~$\div(f)$ sur~$X$.
Pour tout \'el\'ement de~$\bDiv(X)_\Q$ repr\'esent\'e par 
un diviseur de Cartier~$D$ sur un mod\`ele~$X'$ de~$X$,
le diviseur~$D$ est localement sur~$X'$ de la forme~$\div(f)$;
en combinant ces repr\'esentations, on obtient que le morphisme~$\theta$
est surjectif. 
La description des diviseurs de Cartier birationnels
en termes de valuations (divisorielles) entra\^{\i}ne
que le noyau de~$\theta$
est \'egal au noyau de l'homomorphisme canonique
de~$U_\Q(k(X)/k)$ sur~$\Div_\Q(k(X)/k)$.
\end{proof}

Compte tenu de cette proposition, les fonctionnelles divisorielles
sur~$k(X)$ correspondent aux familles $(\lambda_{X'})$ o\`u,
pour tout mod\`ele~$X'$ de~$X$, $\lambda_{X'}$ est une forme lin\'eaire
sur le groupe~$\Div(X')$ des diviseurs de Cartier sur~$X'$,
positive sur le c\^one effectif, ces formes lin\'eaires \'etant astreintes
\`a la relation de compatibilit\'e $\lambda_{X''}(\pi^*D)=\lambda_{X'}(D)$
si $\pi\colon X''\to X'$ est un $X$-morphisme et $D$ un diviseur
de Cartier sur~$X'$. On peut se limiter, si l'on veut,
aux mod\`eles projectifs~$X'$ qui dominent un mod\`ele donn\'e de~$X$.

La th\'eorie de l'intersection fournit alors des exemples g\'eom\'etriques
de fonctionnelles divisorielles. 
Consid\'erons en effet un mod\`ele projectif~$X'$ de~$X$ et soit $A$
un diviseur ample sur~$X'$.
Pour tout $X$-morphisme projectif et birationnel
$\pi\colon X''\to X'$, d\'efinissons une forme lin\'eaire $\lambda_{X''}$
sur~$\Div(X'')_\Q$ par
\[ \lambda_{X''} (D) = \pi^*A^{d-1}\cdot D, \]
o\`u on a pos\'e $d=\dim(X)=\degtr_k(k(X))$.
Ces formes lin\'eaires d\'efinissent une fonctionnelle divisorielle sur~$k(X)$.
De telles fonctionnelles divisorielles seront dites \emph{g\'eom\'etriques}.

\begin{lemm}
Soit $\lambda$ une fonctionnelle divisorielle \emph{globale} sur~$k(X)$.
Pour tout mod\`ele projectif~$X'$, la forme lin\'eaire~$\lambda_{X'}$
sur $\Div(X')$ s'annule sur
le groupe des diviseurs num\'eriquement triviaux.
\end{lemm}
\begin{proof}
Soit $D$ un diviseur de Cartier num\'eriquement trivial sur~$X'$. 
Soit $A$ un diviseur ample sur~$X'$.
D'apr\`es le crit\`ere de Kleiman, pour tout entier relatif~$n$,
il existe un entier~$q\geq 1$ tel que le diviseur $q(A+nD)$ soit lin\'eairement
\'equivalent \`a un diviseur effectif;
soit $f\in k(X)^\times$ tel que $q(A+nD)-\div(f)$ soit effectif.
Comme $\lambda$ est globale, on a $\lambda (\div(f))=0$,
de sorte que $\lambda_{X'}(A+nD)\geq 0$.
Ainsi, $n \lambda_{X'}(D) \geq -\lambda_{X'}(A)$.
Ceci vaut pour tout $n\in\Z$; en faisant tendre~$n$ vers~$\pm\infty$,
on obtient $\lambda_{X'}(D)=0$, ce qu'il fallait d\'emontrer.
\end{proof}

\begin{rema}
Revenons sur l'exemple~\ref{exem-aw}
en supposant que $X$ est une courbe projective int\`egre et r\'eguli\`ere.
Dans ce cas, $\bDiv(X)_\Q$
est le $\Q$-espace vectoriel de base l'ensemble des points ferm\'es de~$X$.
Soit $\lambda$ une fonctionnelle divisorielle globale sur~$F$
qui est triviale sur~$k$.
D'apr\`es le lemme pr\'ec\'edent, il existe un unique nombre
r\'eel~$c$ tel que $\lambda(D)=c\deg(D)$ pour tout $D\in\bDiv(X)_\Q$.
Par d\'efinition, une fonctionnelle divisorielle est positive
en tout diviseur effectif ; on a donc $c\geq 0$.
\end{rema}

\begin{theo}[\cite{BenYaacovHrushovski-2022}]\label{theo.byh-dense}
Soit $k$ un corps et soit $X$ un $k$-sch\'ema projectif, int\`egre et normal.
L'ensemble des fonctionnelles divisorielles g\'eom\'etriques sur~$k(X)$ est dense
dans l'espace des fonctionnelles divisorielles sur~$k(X)$
qui sont triviales sur~$k$, muni de la topologie faible.
\end{theo}

Compte tenu de la d\'efinition de cette topologie,
et quitte \`a remplacer~$X$ par un mod\`ele projectif convenable,
il s'agit de d\'emontrer qu'une forme lin\'eaire sur l'espace
des $\Q$-diviseurs de Cartier sur~$X$ qui est positive sur le c\^one effectif 
et nulle sur le sous-espace des diviseurs de Cartier 
num\'eriquement \'equivalents \`a z\'ero
peut \^etre approch\'ee
par des formes lin\'eaires du type $D\mapsto (A^{d-1}\cdot D)$,
o\`u $A$ est la classe d'un $\Q$-diviseur ample.
\textcite{BenYaacovHrushovski-2022} d\'eduisent ce r\'esultat
d'un \'enonc\'e plus pr\'ecis.

Soit donc $X$ une vari\'et\'e projective, int\`egre, sur un corps~$k$;
soit $d$ sa dimension. Soit~$Z_1(X)$ le groupe des classes
d'\'equivalence lin\'eaire de cycles de dimension~$1$
sur~$X$ et soit~$\Div(X)$ le groupe des classes d'\'equivalence lin\'eaire
de diviseurs de Cartier sur~$X$.
Le degr\'e d'un diviseur de Cartier sur une courbe propre munit le produit
les espaces vectoriels $\Div(X)_\R$ et $Z_1(X)_\R$ 
d'une forme bilin\'eaire. 

Notons~$N^1(X)$ et $N_1(X)$ les quotients de ces espaces
par les noyaux de cette forme (\'equivalence num\'erique); 
ils sont de dimension finie et la forme bilin\'eaire
induite par le degr\'e identifie chacun au dual de l'autre.

Le \emph{volume} d'un diviseur de Cartier~$D$ sur~$X$ est la limite sup\'erieure
\[ \vol(D) = \varlimsup_{n\to+\infty} (d!/n^d) \dim_k H^0(X, \mathscr O(nD)). \]
C'est en fait une limite.
Sur un corps alg\'ebriquement clos,
ce r\'esultat a d'abord  \'et\'e d\'emontr\'e en caract\'eristique z\'ero
par~\textcite{Fujita-1994}, il a ensuite \'et\'e \'etendu au cas
g\'en\'eral 
par~\textcite{Takagi-2007} via les alt\'erations de~\textcite{deJong-1996},
et par~\textcite{LazarsfeldMustata-2009} par la m\'ethode des \og corps d'Okounkov \fg;
sur un corps arbitraire, la preuve est due \`a~\textcite{Cutkosky-2014}.

La fonction volume s'\'etend en une fonction continue, 
$d$-homog\`ene, sur~$N^1(X)$, encore not\'ee~$\vol$;
la preuve donn\'ee par~\textcite[corollaire~2.2.45]{Lazarsfeld-2004}
vaut sur un corps arbitraire.

L'espace $N^1(X)$ contient deux c\^ones convexes fondamentaux :
\begin{itemize}
\item Le c\^one effectif, engendr\'e par les classes de diviseurs effectifs.
Son adh\'erence est appel\'ee le c\^one \emph{pseudo-effectif}.
Son int\'erieur est le c\^one \emph{gros} : une classe~$\alpha$
est grosse si et seulement si $\vol(\alpha)>0$;
\item Le c\^one ample, engendr\'e par les classes de diviseurs amples;
il est ouvert et son adh\'erence
est le c\^one \emph{num\'eriquement effectif.}
Sur ce c\^one, le volume se calcule par le th\'eor\`eme d'Hilbert--Samuel:
$\vol(\alpha)=(\alpha^d)$; c'est en particulier une application polynomiale.
\end{itemize}

Dans l'espace~$N_1(X)$, on d\'efinit le c\^one pseudo-effectif
comme l'adh\'erence du c\^one convexe engendr\'e par
les classes de courbes irr\'eductibles.

Ces c\^ones sont saillants. On munit alors~$N_1(X)$ (resp.~$N^1(X)$)
de la structure d'espace vectoriel ordonn\'e 
pour laquelle les \'el\'ements positifs sont les \'el\'ements du c\^one
pseudo-effectif.

Pour tout morphisme projectif et birationnel $\pi\colon X'\to X$,
et toute classe grosse $\alpha\in N^1(X)$, la classe~$\pi^*\alpha$
est encore grosse, donc admet une d\'ecomposition $\pi^*\alpha=\beta+\gamma$,
o\`u $\beta$ est ample et $\gamma$ est pseudo-effective;
on a donc $\beta\leq\pi^*\alpha$.
Si $\alpha_1,\dots,\alpha_{d-1}\in N^1(X)$ sont des classes grosses, 
on d\'emontre que lorsque $\pi\colon X'\to X$ parcourt
la classe des morphismes projectifs et birationnels,
et $\beta_1,\dots,\beta_{d-1}\in N^1(X')$ sont des classes amples
telles que
$\beta_{p}\leq \pi^*\alpha_{p}$ pour tout~$p$,
alors les classes de cycle
\[ \pi_*(\beta_1\dots \beta_{d-1} ) \in N_1(X) \]
poss\`edent une borne sup\'erieure ; 
on la note $\langle \alpha_1\dots\alpha_{d-1}\rangle$:
c'est l'\emph{intersection positive} des classes $\alpha_1,\dots,\alpha_{d-1}$.

\begin{prop}
Pour tout $\alpha\in N^1(X)$ dans le c\^one gros, la fonction~$\vol$
est diff\'erentiable en~$\alpha$ et l'on a
\[ D\vol(\alpha)  = d \langle \alpha^{d-1}\rangle. \]
De plus, on a 
\[ \vol (\alpha) = \langle \alpha^{d-1}\rangle \cdot \alpha. \]
\end{prop}
Cette proposition est d\'emontr\'ee par \textcite{BoucksomFavreJonsson-2008}
lorsque le corps~$k$ est alg\'ebriquement clos de caract\'eristique z\'ero;
\textcite{Cutkosky-2015} et \textcite{BenYaacovHrushovski-2022} ont
\'etendu la preuve \`a tout corps~$k$.

Nous pouvons maintenant terminer la d\'emonstration du th\'eor\`eme~\ref{theo.byh-dense}.
Les arguments donn\'es par~\textcite{BenYaacovHrushovski-2022}
anticipent ceux de~\textcite[th\'eor\`eme~C]{DangFavre-2022} 
en caract\'eristique z\'ero.

Soit $\lambda$ une forme lin\'eaire sur~$N^1(X)$ 
qui est positive sur le c\^one pseudo-effectif. Quitte \`a lui ajouter
une multiple d'une forme lin\'eaire du type $\gamma\mapsto \alpha^{d-1}\gamma$,
o\`u $\alpha$ est ample et petit, on peut supposer que $\lambda$
est strictement positive sur tout \'el\'ement non nul du c\^one pseudo-effectif.
La fonction $\alpha\mapsto \vol(\alpha)/\lambda(\alpha)^n$
sur le c\^one gros est continue, homog\`ene, et s'annule au bord de ce c\^one ; 
elle atteint donc son maximum en une classe $\alpha$.
En \'ecrivant que la diff\'erentielle du volume s'annule en~$\alpha$,
on obtient  l'\'egalit\'e
\[ \lambda = \frac{\lambda(\alpha)}{\vol(\alpha)} \langle \alpha^{d-1}\rangle.\]
Par d\'efinition de l'intersection positive, 
la forme lin\'eaire $\langle\alpha^{d-1}\rangle$ est limite de formes
lin\'eaires de la forme $\pi_*\beta^{d-1}$, o\`u $\pi\colon X'\to X$
est un morphisme projectif et birationnel et $\beta$ est une classe
ample sur~$X'$, d'o\`u le th\'eor\`eme.

\begin{rema}
On peut choisir pour $\beta$ des classes de $\Q$-diviseurs amples.
D'apr\`es le th\'eor\`eme de Bertini, $\pi_*(\beta^{d-1})$ est alors repr\'esent\'e
par un cycle de la forme~$w[C]$, o\`u $C$ est une courbe irr\'eductible  sur~$X$
et $w$ est un nombre rationnel strictement positif.
De telles classes de courbes sont \emph{mobiles}:
pour toute sous-vari\'et\'e stricte~$Y$ de~$X$, 
on peut choisir une telle courbe~$C$ qui ne soit contenue
dans aucune composante de~$Y$.
On retrouve ainsi le th\'eor\`eme de~\textcite{BoucksomDemaillyPaunEtAl-2013} que le dual du c\^one pseudo-effectif de~$N^1(X)_\R$
est le c\^one convexe ferm\'e
de~$N_1(X)_\R$ engendr\'e par les classes de courbes mobiles.

En passant \`a la limite, cela fournit \'egalement un hom\'eomorphisme
de l'espace des fonctionnelles divisorielles sur~$k(X)$
sur la limite projective des c\^ones des courbes mobiles dans $N_1(X')_\R$,
o\`u~$X'$ parcourt la classe des mod\`eles projectifs de~$X$
\parencite[th\'eor\`eme~11.6]{BenYaacovHrushovski-2022}.
\end{rema}

\subsection{G\'eom\'etrie d'Arakelov}

Soit $F$ un corps de type fini sur~$\Q$. On souhaite
\'etudier les structures de corps globalement valu\'es sur~$F$
qui induisent la structure standard sur~$\Q$.
Les r\'esultats du paragraphe pr\'ec\'edent supposent que la structure
globalement valu\'ee est triviale sur un sous-corps fix\'e
et sont donc sans effet. 
Dans ce paragraphe, nous d\'ecrivons 
les th\'eor\`emes de~\textcite{Szachniewicz-2023}
compagnons des pr\'ec\'edents.

Le corps~$F$ est le corps des fonctions d'un $\Z$-sch\'ema projectif 
et plat~$\mathscr X$, int\`egre.  
L'analogie classique entre corps de nombres et corps de fonctions
am\`ene \`a consid\'erer ce sch\'ema comme \og fibr\'e\fg au-dessus
de la courbe arithm\'etique $\Spec(\Z)$, mais le caract\`ere affine
de cette derni\`ere est un obstacle au d\'eveloppement
d'une th\'eorie de l'intersection sur~$\mathscr X$. La g\'eom\'etrie d'Arakelov
r\'esout cette question en adjoignant aux objets usuels
de la th\'eorie de l'intersection (diviseurs, fibr\'es en droites, cycles\dots)
des structures analytiques (fonctions de Green, m\'etriques, courants de Green\dots)
sur l'espace analytique complexe $\mathscr X(\C)$.

Initi\'ee dans le cas des surfaces arithm\'etiques par \textcite{Arakelov-1974a}, 
la th\'eorie de l'intersection arithm\'etique a \'et\'e construite 
en dimension arbitraire par~\textcite{GilletSoule-1990};
citons aussi~\textcite{Zhang-1995b} pour l'\'etude de l'amplitude dans ce cadre,
et \textcite{Yuan-2008,Chen-2011} pour celle du volume.

Soit donc $\mathscr X$ un $\Z$-sch\'ema projectif et plat, int\`egre,
surjectif sur~$\Spec(\Z)$.
Pour simplifier la discussion, on le suppose normal. 
Soit $\kappa(\mathscr X)$ le corps des fonctions de~$\mathscr X$
et soit~$d$ son degr\'e de transcendance sur~$\Q$;
on a donc $\dim(\mathscr X)=d+1$.

\begin{defi}
Si $D$ est un diviseur de Cartier sur~$\mathscr X$,
une \emph{fonction de Green} pour~$D$ est
une fonction $g_D\colon \mathscr X(\C)\setminus\abs D\to \R$
telle que pour tout ouvert~$U$ de~$\mathscr X$
et toute fonction m\'eromorphe inversible $f\in\kappa(\mathscr X)^\times$
d\'efinissant~$D$ sur~$U$, la fonction $g_D+\log \abs f$
est la restriction d'une fonction sur $U(\C)$ qui est continue et
invariante par conjugaison complexe.
\end{defi}

Alors, le fibr\'e en droites $\mathscr O_{\mathscr X}(D)$
poss\`ede une unique m\'etrique continue pour laquelle la section 
m\'eromorphe canonique a norme~$e^{-g_D}$ hors de~$\abs D$.

Une fonction de Green pour le diviseur nul est une fonction continue
sur~$\mathscr X(\C)$ et invariante par la conjugaison complexe;
en g\'en\'eral, l'ensemble des fonctions de Green pour~$D$ est un espace
affine sous cet espace vectoriel.

Si $(D,g_D)$ et~$(E,g_E)$ sont des diviseurs de Cartier
sur~$\mathscr X$ munis de fonctions de Green, la fonction $g_D+g_E$
sur $\mathscr X(\C) \setminus (\abs D\cup \abs E)$ se prolonge
de mani\`ere unique en une fonction de Green pour~$D+E$.
Cela revient aussi au produit tensoriel des fibr\'es en droites
munis de m\'etriques continues.

\begin{defi}
Le groupe $\hDiv(\mathscr X)$ est le $\Z$-module libre
de base l'ensemble des couples $(D,g_D)$, o\`u $D$ est un diviseur de Cartier
irr\'eductible sur~$\mathscr X$ et $g_D$ une fonction de Green pour~$D$.
Ses \'el\'ements sont appel\'es \emph{diviseurs arithm\'etiques.}
\end{defi}

On d\'efinit de m\^eme
le $\Q$-espace vectoriel $\hDiv_\Q(\mathscr X)$
et le $\R$-espace vectoriel $\hDiv_\R(\mathscr X)$;
leurs \'el\'ement sont respectivement appel\'es
$\Q$-diviseurs arithm\'etiques et $\R$-diviseurs arithm\'etiques.
L'espace $\hDiv_\Q(\mathscr X)$ s'identifie \`a $\hDiv(\mathscr X)\otimes_\Z\Q$,
mais on prendra garde\footnote{%
Voir la note~12, p.~276, de~\textcite{Bost-1999}.}
que 
$\hDiv_\R(\mathscr X)$ n'est pas \'egal \`a $\hDiv(\mathscr X)\otimes_\Q\R$.

Pour tout $f\in\kappa(\mathscr X)^\times$, 
on d\'efinit un diviseur arithm\'etique en posant 
\[ \hdiv(f)=(\div(f),\log \abs f^{-1}).\]

On en d\'eduit une application lin\'eaire de
$\kappa(\mathscr X)^\times_\R$ dans~$\hDiv_\R(\mathscr X)$,
encore not\'ee $\hdiv$.

On munit l'espace des $\R$-diviseurs arithm\'etiques
de la topologie localement convexe la moins fine 
qui co\"{\i}ncide avec la topologie localement convexe canonique 
sur ses sous-espaces vectoriels de dimension finie, 
et qui induit la topologie de la convergence compacte 
sur le sous-espace $\mathscr C(\mathscr X(\C)/\mathord\sim;\R)$.

On dit qu'un \'el\'ement $(D,g_D)$ de $\hDiv_\R(\mathscr X)$ est \emph{effectif}
si $D\geq 0$ et $g_D\geq 0$. 
Les diviseurs arithm\'etiques effectifs
forment un c\^one convexe ferm\'e dans~$\hDiv_\R(\mathscr X)$.

\subsubsection*{Volume arithm\'etique}

Pour tout $(D,g_D)\in\hDiv_\R(\mathscr X)$, on pose
\[ \widehat H^0(\mathscr X, (D,g_D)) = \{ f\in\kappa(\mathscr X)\sozat
      (D,g_D) + \hdiv(f)\geq 0 \} \cup \{0\}. \]

Lorsque $(D,g_D)\in\hDiv(\mathscr X)$,
cet ensemble correspond \`a celui des sections globales du fibr\'e
en droites $\mathscr O_{\mathscr X}(D)$ de norme au plus~$1$ en tout point,
pour la m\'etrique continue correspondant \`a~$g_D$, c'est-\`a-dire
la boule unit\'e d'un espace vectoriel norm\'e de dimension finie;
il est en particulier fini.

On en d\'eduit que cette propri\'et\'e de finitude vaut en g\'en\'eral, 
et l'on pose
\[ \widehat h^0(\mathscr X, (D,g_D)) 
= \log (\Card(\widehat H^0(\mathscr X,(D,g_D)))). \]
On d\'efinit alors le \emph{volume arithm\'etique} de $(D,g_D)$ par l'expression
\begin{equation}
 \hvol (D,g_D) = \varlimsup_n \frac{(d+1)!}{n^{d+1}} \widehat h^0(\mathscr X, n(D,g_D)).
\end{equation}
Cette expression est finie et ne d\'epend que de la classe de $(D,g_D)$
modulo les diviseurs de la forme $\hdiv(f)$.
Lorsque $(D,g_D)\in\hDiv_\Q(\mathscr X)$, \textcite{Chen-2010} a d\'emontr\'e 
que cette limite sup\'erieure est une limite.
En outre, la fonction~$\hvol$ est continue sur $\hDiv_\R(\mathscr X)$
\parencite[th\'eor\`eme~4.4]{Moriwaki-2009b}
et homog\`ene de degr\'e~$(d+1)$.

Si $(E,g_E)$ est effectif, 
on a l'in\'egalit\'e $\hvol(D,g_D)\leq \hvol(D+E,g_D+g_E)$ 
pour tout diviseur arithm\'etique $(D,g_D)$.

On dit que le diviseur arithm\'etique~$(D,g_D)$ est gros
si l'on a $\hvol(D,g_D)>0$ ; on dit qu'il est pseudo-effectif
si $(D,g_D)+\widehat B$ est gros pour tout diviseur arithm\'etique~$\widehat B$
qui est gros.

Dans certains cas, ces volumes arithm\'etiques s'epxriment
en termes de la th\'eorie de l'accouplement d'intersection arithm\'etique. 
Une difficult\'e est que cette th\'eorie n'est pas d\'efinie sur tout $\hDiv_\R(\mathscr X)$.

\subsubsection*{Th\'eorie de l'intersection arithm\'etique}
Un diviseur arithm\'etique $(D,g_D)$ est dit \emph{lisse}
si pour tout ouvert~$U$ de~$\mathscr X$ et toute fonction 
m\'eromorphe inversible~$f$ d\'efinissant~$D$ sur~$U$,
la fonction $g_D+\log \abs f$ est la restriction
d'une fonction lisse sur~$U(\C)$. Cela entra\^{\i}ne que $g_D$
est lisse sur $\mathscr X(\C)\setminus \abs D$
et que la $(1,1)$-forme $\ddc g_D$ sur cet ouvert se prolonge
en une forme lisse sur~$\mathscr X(\C)$, que nous noterons~$\omega_{(D,g_D)}$.
Pour toute $f\in\kappa(\mathscr X)^\times$,
le diviseur arithm\'etique $\hdiv(f)$ est lisse,
et la forme $\omega_{\hdiv(f)}$ est nulle.
En g\'en\'eral, si $(D,g_D)$ est un diviseur arithm\'etique lisse,
la formule de Poincar\'e--Lelong
\[ \ddc g_D + \delta_D = [\omega_{(D,g_D)}] \]
exprime que $g_D$ est un \emph{courant de Green} pour le cycle~$D$
au sens de~\textcite{GilletSoule-1990}.

La th\'eorie de l'intersection arithm\'etique 
d\'efinie par~\textcite{GilletSoule-1990} 
permet de munir l'espace $\hDiv(\mathscr X)$ 
des diviseurs arithm\'etiques lisses sur~$\mathscr X$
d'une application $(d+1)$-lin\'eaire sym\'etrique \`a valeurs
r\'eelles, nulle sur les diviseurs du type $\hdiv(f)$.
En fait, \textcite{GilletSoule-1990} construisent
une th\'eorie g\'en\'erale d'intersection de \og cycles arithm\'etiques \fg,
mais supposent la r\'egularit\'e de~$\mathscr X$ ; 
quitte \`a consid\'erer des coefficients rationnels,
l'utilisation d'alt\'erations \parencite{deJong-1996} permet d'\'eviter
cette hypoth\`ese.

Dans le cas des diviseurs, \textcite{Faltings-1992} avait montr\'e 
comment la formule de r\'ecurrence sous-jacente 
\`a la construction de~\textcite{GilletSoule-1990} permet 
\'egalement d'\'eviter cette hypoth\`ese.
Soit $(D_0,g_0),\dots,(D_d,g_d)$ des $\R$-diviseurs arithm\'etiques;
pour tout~$p$, notons $\omega_p=\omega_{(D_p,g_p)}$.
Lorsque $D_0$ est contenu dans 
la fibre de~$\mathscr X$ au-dessus de l'id\'eal maximal~$\langle p\rangle$
de~$\Z$, l'intersection arithm\'etique des~$(D_i,g_i)$ est donn\'ee par
\[ (D_0,g_0)\cdots (D_d, g_d) = 
  \deg (D_1|_{D_0} \cdot D_d |_{D_0}) \log(p) + \int_{\mathscr X(\C)} g_0 
   \omega_1 \wedge \cdots \wedge \omega_d ,\] 
formule dans laquelle les restrictions~$D_p|_{D_0}$
des $\R$-diviseurs de Cartier~$D_p$ \`a~$D_0$
est d\'efinie \`a \'equivalence lin\'eaire pr\`es.
En particulier, on a 
\[ (0,g_0)(D_1,g_1)\cdots (D_d,g_d)
 = \int_{\mathscr X(\C)} g_0 
   \omega_1 \wedge \cdots \wedge \omega_d .\] 
Lorsque $D_0$ est irr\'eductible, surjectif sur~$\Spec(\Z)$,
et n'est contenu dans aucune composante de~$D_1,\dots,D_d$,
on a
\[ (D_0,g_0)\cdots (D_d, g_d) = 
   ((D_1,g_1)|_{D_0})\cdots (D_d,g_D)|_{D_0}
      + \int_{\mathscr X(\C)} g_0 
   \omega_1 \wedge \cdots \wedge \omega_d .\] 
Dans le cas g\'en\'eral, il faut modifier $(D_0,g_0)$
en lui ajoutant un diviseur arithm\'etique de la forme~$\hdiv(f)$
pour se ramener \`a cette situation.

Cette formule montre \'egalement que l'on peut d\'efinir
cette intersection arithm\'etique lorsque au plus un des $\R$-diviseurs
arithm\'etiques consid\'er\'es (mis en premi\`ere position) n'est pas lisse.

On dit qu'un diviseur arithm\'etique lisse 
$(D,g_D)\in\hDiv(\mathscr X)$ est \emph{nef}
si le $\R$-diviseur de Cartier~$D$ est relativement nef et si
la forme~$\omega_{(D,g_D)}$ est positive. 
Dans ce cas, \textcite[Corollaire~5.5]{Moriwaki-2009b} a d\'emontr\'e
l'\'egalit\'e suivante (\og Hilbert-Samuel arithm\'etique \fg):
\begin{equation}
\hvol (D,g_D) = (D,g_D)^{d+1}.
\end{equation}

\subsubsection{Diff\'erentiabilit\'e du volume arithm\'etique}
Le th\'eor\`eme suivant est l'analogue arithm\'etique des propri\'et\'es
du volume en g\'eom\'etrie alg\'ebrique.
Il est d\^u \`a \textcite{Chen-2011} dans le cas des~$\Q$-diviseurs,
et \`a \textcite{Ikoma-2015} en g\'en\'eral.
Leurs d\'emonstrations sont inspir\'ees par le th\'eor\`eme g\'eom\'etrique.
Il s'agit en particulier de d\'efinir un analogue 
de l'intersection positive pour les diviseurs arithm\'etiques.

Si $(D_0,g_0),\dots,(D_d,g_d)$ sont des $\R$-diviseurs arithm\'etiques
tels que $(D_{k},g_{k})$ est lisse et nef pour $k\geq p+1$,
leur \emph{intersection positive}
\[ \langle (D_1,g_1)\cdots (D_p,g_p)\rangle \cdot
    (D_{p+1},g_{p+1})\cdots (D_d,g_d) \]
est d\'efinie comme la borne sup\'erieure des intersections
\[ (D'_1,g'_1)\cdots (D'_p,g'_p) \pi^*(D_{p+1},g_{p+1})\cdots \pi^*(D_d,g_d)\]
o\`u $\pi$ parcourt la classe des morphismes projectifs
et birationnels $\pi\colon\mathscr X'\to\mathscr X$
et o\`u, pour $k\in\{0,\dots,p\}$,
$(D'_k,g'_k)$ est un $\Q$-diviseur arithm\'etique sur~$\mathscr X'$
tel que $(D'_k,g'_k)\leq \pi^*(D_k,g_k)$.
Cette expression est multi-additive en $(D_k,g_k)$ pour $k\geq p+1$
et s'\'etend de mani\`ere unique en une forme multilin\'eaire sur l'espace
vectoriel des diviseurs arithm\'etiques lisses.
Lorsque $k=d$, elle s'\'etend de mani\`ere unique en une forme lin\'eaire
continue sur $\hDiv_\R(\mathscr X)$.

\begin{theo}
Soit $(D,g_D)\in\hDiv_\R(\mathscr X)$ un $\R$-diviseur arithm\'etique gros
et soit $(E,g_E)\in\hDiv_\R(\mathscr X)$ un $\R$-diviseur arithm\'etique.
La fonction  $t\mapsto \hvol ((D,g_D)+t(E,g_E))$ est d\'erivable en $t=0$,
de d\'eriv\'ee
\[ (d + 1) \langle (D,g_D)^d\rangle \cdot (E,g_E). \]
On a de plus:
\[ \hvol(D,g_D) = \langle (D,g_D)^{d+1}\rangle = \langle (D,g_D)^d\rangle \cdot (D,g_D). \]
\end{theo}

\subsection{G\'eom\'etrie d'Arakelov birationnelle}

Soit $F$ un corps de type fini sur~$\Q$; notons~$d=\degtr_\Q(F)$.
Soit $\hDiv_\Q(F)$ la limite inductive des espaces vectoriels totalement
ordonn\'es
$\hDiv_\Q(\mathscr X)$, lorsque $\mathscr X$ parcourt la classe ordonn\'ee,
filtrante \`a gauche,
des \emph{mod\`eles} de~$F$,
c'est-\`a-dire des 
$\Z$-sch\'emas projectifs et plats~$\mathscr X$, int\`egres et normaux,
munis d'un isomorphisme $F\simeq \kappa(\mathscr X)$.
Il est r\'eticul\'e: la d\'emonstration est analogue au cas g\'eom\'etrique,
il faut en outre observer que la borne inf\'erieure de deux fonctions de Green
fournit une fonction de Green pour la borne inf\'erieure
des diviseurs de Cartier correspondants.

Pour tout $f\in F\times$,
les diviseurs arithm\'etiques $\hdiv(f)$ sur les diff\'erents mod\`eles~$\mathscr X$
fournissent un \'el\'ement de~$\hDiv_\Q(\mathscr X)$
que nous notons encore $\hdiv(f)$.

\begin{prop}
Il existe un unique morphisme d'espaces vectoriels r\'eticul\'es
\[ \Div_\Q(F) \to \hDiv_\Q(\mathscr X) \]
qui applique $\div(f)$ sur $\hdiv(f)$ pour tout $f\in F^\times$.
Ce morphisme induit une bijection entre
fonctionnelles divisorielles~$\lambda$ sur~$F$
qui induisent la structure canonique sur~$\Q$
et
formes lin\'eaires $\lambda\colon \hDiv_\R(\mathscr X)\to\R$
qui sont positives en tout diviseur arithm\'etique effectif,
nulles en tout diviseur arithm\'etique principal
et v\'erifient $\lambda((0,1))=1$.
\end{prop}

\begin{rema}
Revenons \`a l'exemple~\ref{exem-aw} en supposant $F=\Q$.
La description pr\'ec\'edente conduit \`a poser $\mathscr X=\Spec(\Z)$.
Dans ce cas, l'intersection arithm\'etique
correspond au \og degr\'e arithm\'etique\fg  $\hdeg\colon \hDiv(\mathscr X)\to\R$
donn\'e par $\hdeg(\langle p\rangle, g) = \log(p)+g$
pour tout nombre premier~$p$ et tout nombre r\'eel~$g$.
Si $(D,g_D)$ est un $\R$-diviseur arithm\'etique
effectif, son degr\'e arithm\'etique est positif ou nul;
de plus, $\hdeg(\hdiv(f))=0$ pour tout $f\in\Q$.
Cette application induit alors un isomorphisme
de $\hDiv_\R(\mathscr Z)/\langle \hdiv(\Q^\times)\rangle$ sur~$\R$,
qui correspond \`a la structure de hauteur standard sur~$\Q$.
Toute fonctionnelle divisorielle sur $\Q$ en est un multiple positif.
\end{rema}

On dit qu'un diviseur arithm\'etique lisse $(A,g_A)$ sur~$\mathscr X$ 
est \emph{ample}\footnote{Les d\'efinitions fluctuent dans la litt\'erature;
nous prenons ici celle de~\textcite[d\'efinition~2.3]{Charles-2021}.}
si $A$ est ample, si la forme $\omega_{(A,g_A)}$ est strictement
positive en tout point, 
et si pour tout entier~$n$ assez grand, 
le fibr\'e en droites $\mathscr O_{\mathscr X}(nA)$, muni
de sa m\'etrique associ\'ee \`a~$ng_A$, 
est engendr\'e par ses sections globales de norme~$<1$ en tout point.

On dit qu'un $\R$-diviseur arithm\'etique est \emph{nef} (resp.\ \emph{ample})
si c'est une combinaison lin\'eaire (non vide) 
\`a coefficients strictement positifs de diviseurs arithm\'etiques 
nef (resp.\ amples).

Par la th\'eorie de l'intersection arithm\'etique,
un $\R$-diviseur lisse $(A,g_A)$ sur~$\mathscr X$
fournit une forme lin\'eaire sur $\hDiv_\R(\mathscr X)$,
par la formule 
\[ (D,g_D) \mapsto \frac1{\vol(A_\Q)}(A,g_A)^{d} (D,g_D). \]
Elle s'annule sur tout diviseur arithm\'etique principal.
Lorsque $(A,g_A)$ est ample, cette forme lin\'eaire
est positive sur tout diviseur arithm\'etique effectif.
On a $(A,g_A)^d\cdot (0,1) = (A_\Q)^d=\vol(A_\Q)$.

Par ailleurs, cette forme lin\'eaire s'\'etend en une forme
lin\'eaire sur~$\hDiv_\R(\mathscr X')$, pour tout mod\`ele~$\mathscr X'$
de~$F$ qui domine~$\mathscr X$ et v\'erifie les m\^emes propri\'et\'es:
annulation sur les diviseurs arithm\'etiques principaux,
et positivit\'e en les diviseurs arithm\'etiques effectifs.
Elle d\'efinit ainsi une fonctionnelle divisorielle.
De telles fonctionnelles divisorielles seront dites \emph{arithm\'etiques}.

\begin{theo}[\cite{Szachniewicz-2023}] \label{theo.s-dense}
Soit $F$ un corps de type fini sur~$\Q$.
Les fonctionnelles  divisorielles arithm\'etiques sur~$F$ 
sont denses dans l'espace des fonctionnelles divisorielles
qui induisent la structure standard sur~$\Q$.
\end{theo}
La d\'emonstration est analogue \`a celle que nous avons esquiss\'ee,
en rempla\c{c}ant les objets g\'eom\'etriques (diviseurs, volume)
par les analogues arithm\'etiques 
que nous avons d\'ecrits au paragraphe pr\'ec\'edent.

\section{Logique continue}

\subsection{Le langage de la g\'eom\'etrie diophantienne}

\subsubsection*{Structures}
Dans la version \og born\'ee \fg
de \parencite{BenYaacovBerensteinHensonEtAl-2008} de la logique
continue, les structures sont des espaces m\'etriques compacts
et les pr\'edicats sont des fonctions \`a valeurs
dans un intervalle born\'e, disons~$[0;1]$.
Dans la version non born\'ee \parencite{BenYaacov-2008},
les structures sont des espaces m\'etriques complets~$M$ munis d'une \og jauge \fg,
c'est-\`a-dire une application 1-lipschitzienne $\nu\colon M\to\R$ telle
que la distance soit born\'eee sur chaque partie o\`u~$\nu$ est born\'ee.
Il s'agit ici des corps globalement valu\'es, consid\'er\'es
comme espaces m\'etriques pour la distance triviale,
munis de la fonction jauge $\Ht \colon  x\mapsto h(1, x)$.

\subsubsection*{Fonctions et termes}
Dans cette logique, un symbole de fonction $n$-aire~$f$ repr\'esente
une application de~$M^n$ dans~$M$ qui v\'erifie une condition
d'uniforme continuit\'e uniforme sur chaque partie o\`u~$\nu$ born\'ee,
c'est-\`a-dire qu'on fixe, pour tout tel symbole de fonction~$f$
un \og module \fg, c'est-\`a-dire une fonction $\delta_f\colon\R_+\to\R_+$,
strictement croissante, continue \`a droite et nulle en~$0$,
et on les interpr\'etations de~$f$ seront des fonctions de~$M^n$ dans~$M$
telles que pour tous $x,y\in M^n$ tels que $d(x,y)<r$ et $\nu(x),\nu(y)<1/r$,
on a $d(f(x),f(y))<\delta_f(r)$ et $\nu(f(x,y))<1/\delta_f(r)$.
Dans le cas qui va nous int\'eresser, la condition d'uniforme continuit\'e
s'\'evanouit  et il ne reste qu'une majoration, qu'on r\'ecrit sous la forme
$\nu(f(x,y))<\Delta_f(\sup(\nu(x),\nu(y)))$,
o\`u $\Delta_f\colon\R_+\to\R_+$ est une fonction croissante.
Les symboles de fonctions de la th\'eorie des corps globalement
valu\'es sont simplement les deux symboles de constantes~$0$ et~$1$
et les trois symboles arithm\'etiques $+,-,\cdot$
repr\'esentant bien s\^ur l'addition, la soustraction et la multiplication.
Le lemme ci-dessous montre que l'on peut prendre
pour modules les fonctions donn\'ees par 
$\Delta_0(r)=0$, $\Delta_1(r)=0$, $\Delta_+(r)=\Delta_-(r)=2r+e$,
et $\Delta_\cdot(r)=2r$.
\begin{lemm}
Soit $h$ une hauteur sur un corps~$F$. Pour $x\in F$, posons $\Ht(x)=h_2(1,x)$.
Soit $e$ un nombre r\'eel tel que $e\geq \Ht(2)$.
\begin{enumerate}[\itshape a\textup)]
\item On a $\Ht(0)=0$ et $\Ht(1)=0$;
\item Pour tous $x, y\in F$, on a $\Ht(x+y) \leq \Ht(x)+\Ht(y)+e$
et $\Ht(x-y)\leq\Ht(x)+\Ht(y)+e$;
\item Pour tous $x, y\in F$, on a $\Ht(xy) \leq \Ht(x)+\Ht(y)$.
% \item Pour tout $x\in F^n$, on a $h_n(x) \leq \sum_{i} \Ht(x_i)$.
\end{enumerate}
\end{lemm}

Les \emph{termes} de la logique continue 
repr\'esentent les fonctions de puissances~$M^n$ dans~$M$
que l'on peut d\'efinir dans le langage donn\'e.
Ils sont d\'efinis de mani\`ere usuelle,
par r\'ecurrence,
\`a partir de symboles de variables et des symboles de fonctions.
Ici, ce seront des expressions formelles (bien form\'ees) utilisant 
les symboles~$+$, $-$, $\cdot$ et des symboles de variables,
c'est-\`a-dire --- essentiellement --- des polyn\^omes.

\subsubsection{Pr\'edicats et formules}
En logique continue, les pr\'edicats sont des fonctions \`a valeurs r\'eelles, 
et on impose de mani\`ere analogue aux fonctions que chaque pr\'edicat~$p$ soit muni
d'un module de continuit\'e uniforme et d'une majoration
de la jauge sur chaque partie de la structure o\`u la jauge est born\'ee.
Dans de le cas d'un espace muni de la distance triviale,
cette condition se simplifie en une relation
du type $\Ht(x) < r \Rightarrow  \abs{p(x)}\leq \Delta_p(r)$,
o\`u $\Delta_p\colon \R_+\to\R_+$ est une fonction croissante.

Les symboles de pr\'edicats de la th\'eorie des corps globalement
valu\'es contiennent d'abord un pr\'edicat
binaire~$=$ repr\'esentant l'\'egalit\'e.
On introduit \'egalement, pour tout entier~$n$,
un symbole de pr\'edicat $n$-aire $h_n$, qui repr\'esente la fonction~$h_n$
dans un corps globalement valu\'e.
Si $F$ est un corps globalement valu\'e et $x$ est un \'el\'ement non nul de~$F^n$,
on a $h_n(x) \leq \sum_{i} \Ht(x_i)$ pour tout $x\in F^n$;
cela montre qu'on peut prendre pour module de~$h_n$
la fonction donn\'ee par $\Delta_{h_n}(r)=nr$. 
Cependant, comme 
les axiomes des corps globalement valu\'es ont postul\'e $h_n(0)=-\infty$,
il est plus naturel de consid\'erer ici que les pr\'edicats
sont des fonctions continues \`a valeurs dans~$[-\infty;+\infty\mathclose[$;
nous tairons d\'esormais cette subtilit\'e.
Plus g\'en\'eralement, pour tout polyn\^ome tropical~$t$ en $n$~symboles
de variables, on introduit un symbole de pr\'edicat $n$-aire~$R_t$
qui repr\'esente le terme local associ\'e \`a~$t$.
On d\'efinit le module~$\Delta_t$ du pr\'edicat~$R_t$ de r\'ecurrence
sur la construction du polyn\^ome tropical~$t$ de sorte que
lorsque $F$ est un corps globalement valu\'e et $(x_1,\dots,x_n)\in (F^\times)^n$,
on ait
\[ R_t(x_1,\dots,x_n) \leq \Delta_t(\sup(\Ht(x_1),\dots,\Ht(x_n)).\]
En fait, le symbole~$h_n$ introduit pr\'ec\'edemment peut \^etre pris
comme l'abr\'eviation de~$R_t$ pour $t=-\inf(x_1,\dots,x_n)$.

% Cependant, l'\'egalit\'e $h_n(0)=-\infty$ impos\'ee 
% par la d\'efinition~\ref{defi.h}, et naturelle dans le d\'eveloppement
% de la th\'eorie des corps globalement valu\'es, fait sortir de ce cadre;
% on convient de modifier la valeur de $h_n(0)$,
% par exemple en posant $h_n(0)=0$,
% et d'ajuster les axiomes en cons\'equence.
% Les in\'egalit\'es suivantes assurent le contr\^ole requis de la jauge de ces symboles de fonctions et de pr\'edicats et montrent que l'on peut prendre
% $\Delta_+=\Delta_-\colon r\mapsto 2r+e$,  $\Delta_\cdot\colon r\mapsto 2r$,
% $\Delta_{h_n}\colon r\mapsto r$, $\Delta_0=\Delta_1=0$.
% \begin{lemm}
% Soit $h$ une hauteur sur un corps~$F$. Pour $x\in F$, posons $\Ht(x)=h_2(1,x)$.
% Soit $e$ un nombre r\'eel tel que $e\geq \Ht(2)$.
% \begin{enumerate}[\itshape a\textup)]
% \item On a $\Ht(0)=0$ et $\Ht(1)=0$;
% \item Pour tous $x, y\in F$, on a $\Ht(x+y) \leq \Ht(x)+\Ht(y)+e$
% et $\Ht(x-y)\leq\Ht(x)+\Ht(y)+e$;
% \item Pour tous $x, y\in F$, on a $\Ht(xy) \leq \Ht(x)+\Ht(y)$;
% \item Pour tout $x\in F^n$, on a $h_n(x) \leq \sum_{i} \Ht(x_i)$.
% \end{enumerate}
% \end{lemm}

Les \emph{formules} sont \'egalement d\'efinies par r\'ecurrence:
\begin{itemize}
\item Les formules atomiques sont celles de la forme $pt_1\dots t_n$,
o\`u $p$ est un symbole de pr\'edicat $n$-aire
et $t_1,\dots,t_n$ sont des termes ;
\item Les connecteurs logiques sont fournis par les
fonctions continues de~$\R^n$ dans~$\R$; ainsi,
si $u$ est une telle fonction et si $f_1,\dots, f_n$ sont des formules,
l'expression $uf_1\dots f_n$ est une formule;
\item Le r\^ole des quantificateurs logiques ($\forall$, $\exists$)
est respectivement jou\'e par $\inf$ et $\sup$; 
% il faut assurer l'existence de ces expressions.  Ainsi,
pr\'ecis\'ement, si
 $f$ est une formule, $L$ un ensemble de symboles de variables
et $\phi\colon \R_+\to\R$ une fonction continue \`a support compact,
alors les expressions
$\inf_{L,\phi}$ et $ \sup_{L,\phi} f$  sont des formules.
(Nous expliquerons ci-dessous leur interpr\'etation.)
\end{itemize}
Une formule sans quantificateurs est une formule sans~$\sup$ et~$\inf$.
Les variables libres d'une formule sont d\'efinies de mani\`ere usuelle;
une formule sans variable libre est appel\'ee un \emph{\'enonc\'e}.

\subsection{Types et mod\`eles}

\subsubsection{Interpr\'etation}
Une \emph{structure} pour ce langage est 
simplement un ensemble~$M$, muni d'une fonction $\Ht\colon M\to \R$,
d'\'el\'ements~$0^M$ et~$1^M$ de~$M$ tels que $\Ht(0^M)\leq\Delta_0$ et
$\Ht(1^M)\leq\Delta_1$,
de fonctions $+^M,-^M,\cdot^M$ de~$M^2$ dans~$M$ telles
que $\Ht(x+^My)\leq \Delta_+(r)$ pour $x,y\in M$ tels que $\Ht(x),\Ht(y)<r$, etc.
et de $R_t^M\colon M^n\to \R$, pour tout polyn\^ome tropical~$t$,
telles que $\Ht(R_t^M(x))\leq \Delta_t(r)$ si $\Ht(x_i)<r$ pour tout~$i$.
Lorsque $t=-\inf(x_1,\dots,x_n)$, on \'ecrit $R_t=h_n$.

On commettra souvent l'abus d'\'ecriture consistant \`a \'ecrire~$0$,
$+$, $R_t$ ou~$h_n$ pour~$0^M$, $+^M$, $R_t^M$ ou~$h_n^M$.  

Donnons-nous une telle structure~$M$.

Les termes s'interpr\'etent comme des fonctions sur~$M$ et ses puissances.
Si $t$ est un terme et si $V$ est un ensemble de symboles  de variables
contenant ceux qui apparaissent dans~$t$, 
on d\'efinit par r\'ecurrence une fonction $t^M\colon M^V\to M$,
ainsi qu'une fonction $\Delta_t\colon \R_+\to\R_+$,
croissante et continue \`a droite,
telle que pour tout $a\in M^V$
tel que $\Ht(a_v)<r$ pour tout~$v$, on ait $\Ht(t^M(a))\leq \Delta_t(r)$.
Explicitement, si $t=v$, 
alors $v\in V$ et $t^M$ est la fonction $a\mapsto a_v$;
si $t=t_1+t_2$, alors $t^M$ est la fonction $a\mapsto t_1^M(a)+t_2^M(a)$,
etc.

De m\^eme, les pr\'edicats s'interpr\`etent comme des fonctions \`a valeurs
r\'eelles sur~$M$ et ses puissances.

\`A premi\`ere vue, le pr\'edicat binaire pour l'\'egalit\'e est \`a valeurs bool\'eennes,
mais on interpr\`ete les valeurs de v\'erit\'e vrai et faux respectivement
comme 0 et 1. 

L'interpr\'etation des formules atomiques est assez claire,
de m\^eme que celle des connecteurs logiques.
Il faut cependant pr\'eciser l'interpr\'etation des quantificateurs
$\sup_{L,\phi}$ et $\inf_{L,\phi}$ et, en particulier,
expliquer le r\^ole de la fonction~$\phi$.
Soit donc $f$ une formule dont les variables libres sont 
contenues dans un ensemble~$V$ de symboles de variables,
soit $L$ un ensemble de symboles de variables,
et soit $\phi\colon\R_+\to\R$ une fonction continue \`a support compact.
Pour tout $a\in M^V$ et tout $b\in M^L$,  notons $a[b]$
l'\'el\'ement~$c$ de~$M^V$ tel que $c_v=b_v$ si $v\in L$ et $c_v=a_v$ sinon.
Alors, l'interpr\'etation de la formule $g=\sup_{L,\phi} f$ est d\'efinie par
\[ g^M(a) = \sup_{b\in M^L} \phi(\sup_i\Ht(b_i)) f^M(a[b]) \]
pour tout $a\in M^V$.
La condition de support compact sur~$\phi$ 
garantit  garantit que cette borne sup\'erieure est finie
et plus pr\'ecis\'ement que l'on a une majoration
en termes de $\Ht(a)$ si l'on dispose d'une majoration
similaire pour $f^M(a[b])$. 

Plus g\'en\'eralement, si $f$ est une formule et si $V$ est un ensemble de symboles
de variables contenant ceux qui apparaissent dans~$f$,
on d\'efinit par r\'ecurrence une fonction $f^M\colon M^V\to\R$
ainsi qu'une fonction croissante $\Delta_f\colon\R_+\to\R_+$
telle que  $p(x)\leq\Delta_p(r)$ si $\Ht(x_v)<r$ pour tout $v\in V$.

\subsubsection{R\'ealisation}
Une \emph{th\'eorie} est un ensemble d'\'enonc\'es.

Si une formule~$f$ est un \'enonc\'e, on peut prendre $V=\emptyset$
pour son interpr\'etation dans une structure~$M$,
de sorte que $f$ d\'efinit un nombre r\'eel~$f^M$.
On dit qu'une structure~$M$ est un \emph{mod\`ele} d'une th\'eorie~$T$
si on a $f^M=0$ pour tout \'enonc\'e~$f\in T$.

Comme tout intervalle ferm\'e de~$\R$, 
plus g\'en\'eralement toute partie ferm\'ee~$A$ de~$\R^n$, 
est l'ensemble des z\'eros d'une fonction continue,
on peut faire usage d'\'enonc\'es du type $(p_1,\dots,p_n)\in A$,
o\`u $A$ est une partie ferm\'ee de~$\R^n$. Il suffit
de l'\'ecrire $d_A(p_1,\dots,p_n)=0$ o\`u $d_A\colon\R^n\to \R$
est la distance \`a~$A$ (par exemple pour la norme euclidienne).

De m\^eme que l'\'egalit\'e \`a valeurs bool\'eenne a \'et\'e interpr\'et\'ee
on peut aussi utiliser les connecteurs logiques usuels:
$\vee$ (\textsc{ou}) et $\wedge$ (\textsc{et}).
Si $f$ et $g$ sont deux formules, la formule $f\wedge g$ 
doit \^etre vraie dans~$M$ si et seulement si $f^M$ et $g^M$ sont vraies;
on pose donc $(f\wedge g)^M=\sup(\abs{f^M},\abs{g^M})$.
De m\^eme, la formule $f\vee g$, interpr\'et\'ee
comme $(f\vee g)^M=\inf(f^M ,g^M)$, est vraie dans~$M$
si et seulement si $f^M$ ou $g^M$ est vraie.

Comme $\R_+^*$ n'est pas ferm\'e dans~$\R$,
la n\'egation $\neg f$ d'une formule ou l'implication $f\Rightarrow g$ 
entre deux formules sont un peu plus d\'elicates \`a consid\'erer.
On peut les remplacer par une conjonction infinie de formules. 
Pour tout entier~$n$, soit $\phi_n$ la fonction 
continue \`a support compact sur~$\R$ don\'ee par $\phi_n(t)=\sup(1-n\abs t,0)$.
Par construction, $f^M\neq 0$ \'equivaut \`a la conjonction 
des \'egalit\'es $\phi_n(f)^M=0$.

Les formules universellement quantifi\'ees,
du type \og $\forall x, f(x)$ \fg, o\`u $f$ est une formule ayant
un symbole de variable libre~$x$, doivent \'egalement \^etre r\'ecrites.
Pour cela, on consid\`ere, pour tout entier~$n$,
la fonction~$\phi_n$ sur~$\R_+$ telle que
$\phi_n(t)=1$ pour $t\in [0;n]$, $\phi_n(t)=n+1-t$
pour $t\in[n;n+1]$ et $\phi(t)=0$ pour $t\geq n+1$.
On constate que l'assertion \og $f^M(a)\leq 0$ pour tout $a\in M$ \fg
\'equivaut \`a la conjonction,
pour tout $n\in \N$,
des conditions
\[ \sup_{a\in M} \phi_n(\Ht(a)) f(a) = 0.  \]

\begin{defi}
Soit $e$ un nombre r\'eel. Dans le langage introduit pr\'ec\'edemment,
la th\'eorie des corps globalement valu\'es (d'erreur archim\'edienne~$e$)
consiste en:
\begin{itemize}
\item Les axiomes des corps ;
\item Les axiomes de la d\'efinition~\ref{defi.h} 
et ceux de la d\'efinition~\ref{defi.tl} ;
\item L'in\'egalit\'e $h_2(1,2)\leq e$;
\item La formule du produit $h_1(x)=0$.
\end{itemize}
\end{defi}
On note \GVF[$e$] cette th\'eorie.
Ses mod\`eles sont exactement les corps globalement valu\'es
d'erreur archim\'edienne au plus~$e$.

Plus g\'en\'eralement, si $k$ est un corps, les corps globalement
valu\'es contenant~$k$ et induisant la structure triviale sur~$k$
peuvent \^etre axiomatis\'es par la variante des axiomes pr\'ec\'edents 
o\`u l'on remplace,
dans l'axiome~(iv) de la d\'efinition~\ref{defi.tl}
la condition  \og pour tout $v\in \QVal(F)$ \fg
par la condition \og pour tout $v\in\QVal(F/k)$ \fg.

M\^eme si ce n'en est aujourd'hui que le r\'esultat le plus \'el\'ementaire,
le th\'eor\`eme de compacit\'e est un pilier 
de la logique du premier ordre, 
et un des objectifs du d\'eveloppement de la logique continue 
\'etait de garantir la possibilit\'e un tel th\'eor\`eme.
Pour le d\'emontrer dans le pr\'esent contexte,
on utilise la construction de corps globalement valu\'es par ultraproduits
et la variante suivante du th\'eor\`eme de \L os.

\begin{prop}
Soit $e$ un nombre r\'eel, 
soit $(M_i)_{i\in I}$ une famille de corps globalement valu\'es
d'erreur archim\'edienne au plus~$e$, soit $\mathfrak u$
un ultrafiltre non principal sur~$I$ et soit $M_{\mathfrak u}$
l'ultraproduit des~$M_i$ relativement \`a l'ultrafiltre~$\mathfrak u$.
Pour toute formule~$f$, tout ensemble de symboles de variables~$V$
contenant les variables libres de~$f$
et tout $a\in \prod_i M_i^V$ d'image~$[a]$ dans~$ M_{\mathfrak u}^V$, on a 
\[ f^{M_{\mathfrak u}}([a])= \lim_{i,\mathfrak u} f^{M_i}(a_i).\]
\end{prop}

\begin{coro}
Soit $V$ un ensemble de symboles de variables
et soit $T$ un ensemble de formules dont les variables libres
appartiennent \`a~$V$.
Soit $r\colon V\to\R_+$ une application.
On suppose que pour toute partie finie~$T_1$ de~$T$
et tout $\eps>0$, il existe un corps globalement valu\'e~$M$,
d'erreur archim\'edienne au plus~$e$,
et $a\in M^V$ tel que $\Ht(a_v)<r_v+\eps$
et $f^M(a)<\eps$ pour tout $f\in T_1$.
Alors, il existe un corps globalement valu\'e~$M$
d'erreur archim\'edienne au plus~$e$ et un \'el\'ement $a\in M^V$
tel que $\Ht(a_v)\leq r_v$ pour tout~$v$
et tel que $f^M(a)=0$ pour tout $f\in T_1$.
\end{coro}

Soit $V$ un ensemble de symboles de variables et
soit $\Phi_V$ l'ensemble des formules dont les variables
libres appartiennent \`a~$V$.
On dit qu'une fonction $\alpha\colon\Phi_V\to\R$ est un \emph{$V$-type}
s'il existe un corps globalement valu\'e~$M$ d'erreur archim\'edienne
au plus~$e$ et un \'el\'ement $a\in M^V$ tel que
$\alpha(f)=f^M(a)$ pour tout $f\in\Phi_V$.
On note $\mathscr S_V$ l'ensemble des $V$-types;
on le munit de la topologie de la convergence simple.

Lorsque $V=\emptyset$, c'est l'espace des th\'eories de corps globalement
valu\'es.

\begin{coro}
Soit $V$ un ensemble.
L'espace $\mathscr S_V$ est localement compact.
Pour tout $r\in\R_+$, le sous-espace des types $\alpha\in\mathscr S_V$
tels que $\Ht(\alpha)\leq r$ est compact.
\end{coro}

\subsection{Structures existentiellement closes}

La th\'eorie des corps ordonn\'es et, dans une moindre mesure, celle des corps,
ont \'et\'e les exemples fondamentaux dans la cr\'eation des premi\`eres d\'efinitions
et des premiers concepts en th\'eorie des mod\`eles, 
en particulier la notion d'extension \'el\'ementaire.
Une extension $K\to L$ de corps (resp.\ de corps ordonn\'es) est \'el\'ementaire si
les formules \`a param\`etres dans~$K$ satisfaites par~$L$
sont exactement celles qui sont satisfaites par~$K$.
Si l'on se restreint aux formules qui ne font intervenir
que des quantificateurs existentiels, on obtient
dans ces deux cas
les notions d'extension alg\'ebriquement close (resp.\ r\'eellement close),
mais le principe est g\'en\'eral.
Si un corps (resp.\ un corps ordonn\'e) est existentiellement clos
dans toute extension, on dit qu'il est existentiellement clos,
et on obtient bien s\^ur les notions de corps alg\'ebriquement clos
(resp.\ r\'eel clos); l\`a encore, le principe est g\'en\'eral.

En logique continue, la formulation de ce principe requiert
une petite variation: on ne demande que des solutions approch\'ees.
\begin{defi}
% Soit $M\to N$ une extension de corps globalement valu\'es.
% 
% \begin{enumerate}
% \item
% On dit que cette extension est \emph{\'el\'ementaire} si pour toute formule~$f$,
% tout ensemble de symboles de variables contenant ceux qui 
% apparaissent dans~$V$,
% et tout $a\in M^V$, on a l'\'egalit\'e $f^M(a)=f^N(a)$.
% 
% \item
Soit $M$ un corps globalement valu\'e.
On dit que $M$ est \emph{existentiellement clos} si pour
toute extension $M\to N$ de corps globalement valu\'es,
pour tous ensembles~$V$ et~$W$ de symboles de variables,
pour toute formule sans quantificateurs~$f$  dont les symboles de variables
appartienennt \`a~$V\cup W$, pour tout $a\in M^V$ et tout $r\in\R_+$,
s'il existe $b\in N^W$ tel que $\Ht(b)\leq r$ et $f^N(a,b)=0$,
alors pour tout $\eps>0$, il existe $b'\in M^W$
tel que $\Ht(b')\leq r+\eps$ et $\abs{f^M(a,b')}<\eps$.
\end{defi}

% Cette formulation permet de garantir un \'enonc\'e analogue
% \`a celui qui vaut en logique classique: pour que $M$
% soit existentiellement clos dans~$N$,
% il faut et il suffit qu'il existe un plongement
% de~$N$ dans une ultrapuissance de~$M$.

Les deux th\'eor\`emes principaux de cet expos\'e affirment 
que les cl\^otures alg\'ebriques des 
classiques corps globaux de la th\'eorie alg\'ebrique des nombres,
consid\'er\'es comme des corps globalement valu\'es, 
sont existentiellement clos.
Ils sont dus \`a \textcite{BenYaacovHrushovski-2022}
pour les corps de fonctions et \`a \textcite{Szachniewicz-2023}
pour les corps de nombres. Nous les \'enoncerons
et esquisserons leurs d\'emonstrations aux paragraphes suivants.
\medskip

Le \emph{Nullstellensatz} de Hilbert,
qui caract\'erise les corps alg\'ebriquement clos~$K$,
affirme que si un syst\`eme d'\'equations polynomiales (en plusieurs
variables) \`a coefficients dans~$K$
poss\`ede une solution dans une extension~$L$ de~$K$,
alors ce syst\`eme poss\`ede d\'ej\`a une solution dans~$K$.
Le \emph{Positivstellensatz} d'Artin et Schreier
est la variante de cet \'enonc\'e pour les corps ordonn\'es;
il caract\'erise les corps r\'eels clos.

Dans le cadre de la th\'eorie des mod\`eles, ces th\'eor\`emes montrent
que la notion de corps existentiellement clos (resp.\ de corps
r\'eel clos) poss\`ede une axiomatisation au premier ordre
qui ne fait pas intervenir d'autres corps (resp.\ d'autres corps ordonn\'es). 

Dans le cas des corps globalement valu\'es, c'est une question ouverte.
\begin{conj}
Soit $e$ un nombre r\'eel.
Il existe une th\'eorie dans le langage des corps globalement valu\'es
dont les corps globalement valu\'es (d'erreur archim\'edienne~$e$)
existentiellement clos sont les mod\`eles.
\end{conj}

De mani\`ere \'equivalente, la th\'eorie \GVF[e] admet-elle une th\'eorie mod\`ele compagne?

Si c'\'etait le cas, tout corps globalement valu\'e existentiellement clos~$K$
serait \'el\'ementairement \'equivalent \`a l'un des trois corps
existentiellement clos suivants :
\begin{itemize}
\item Si $h(1,2)>0$, au corps $\overline\Q$,
muni de la structure de corps globalement valu\'e
obtenue en multipliant la structure standard par $h(1,2)/\log(2)$;
\item
Si $h(1,2)=0$ et $K$ est de caract\'eristique z\'ero,
\`a la cl\^oture alg\'ebrique du corps~$\Q(T)$
muni de la structure standard qui est triviale sur~$\Q$;
\item
Si $h(1,2)=0$ et $K$ est de caract\'eristique $p>0$,
\`a la cl\^oture alg\'ebrique du corps~$\F_p(T)$
muni de la structure standard qui est triviale sur~$\F_p$.
\end{itemize}

Les th\'eor\`emes~\ref{theo.byh} et~\ref{theo.s} entra\^{\i}nent
l'\'enonc\'e plus faible que tout corps globalement valu\'e
se plonge dans une ultrapuissance d'un de ces trois corps.

\subsection{Cl\^oture existentielle: corps de fonctions}

\begin{theo}[\cite{BenYaacovHrushovski-2022}] \label{theo.byh}
Soit $k$ un corps et soit $K$ la la cl\^oture alg\'ebrique
du corps~$k(T)$, muni
de sa structure standard de corps globalement valu\'e
pour laquelle $h(1,T)=1$.
Alors $K$ est un corps valu\'e 
existentiellement clos.
\end{theo}

% La notion de \emph{type sans quantificateur}
% est d\'efinie de mani\`ere analogue \`a celle de \emph{type}, mais en ne 
% consid\'erant que des formules sans quantificateurs.
% L'espace $\mathscr S^\qf_V$ des types sans quantificateurs
% \`a symboles de variables dans un ensemble~$V$ est \'egalement localement
% compact, et la jauge en d\'ecoupe des parties compactes.
% On peut alors reformuler cette d\'efinition en terme d'espaces de types.
% Une extension $M\to N$ induit une application continue 
% de~$\mathscr S^\qf_V(N)$ hom\'eomorphisme de l'espace
% des types sans quantificateurs 
 
On peut supposer que $k$ est alg\'ebriquement clos.
Posons $K=\overline{k(T)}$. On consid\`ere un corps globalement
valu\'e~$L$ qui \'etend~$K$ et un point $b=(b_1,\dots,b_n)\in L^n$;
il s'agit d'approcher le \og type sans-quantificateur\fg de~$b$ 
par celui d'un point $a\in K^n$.

Commen\c{c}ons par \'etudier le cas particulier o\`u les seules formules
consid\'er\'ees ne font intervenir que des param\`etres dans~$k$.
Par \'elimination des quantificateurs dans la th\'eorie~\ACF,
on se ram\`ene \`a la question suivante:
Soit $X$ un $k$-sch\'ema propre, int\`egre et normal muni d'un isomorphisme
$k(b_1,\dots,b_n)\simeq k(X)$,  
de sorte que le point~$b$ appara\^{\i}t comme un point g\'en\'erique de~$X$;
posons $d=\dim(X)$.
Soit $Y$ un ferm\'e de Zariski strict de~$X$.
D'apr\`es le th\'eor\`eme~\ref{theo.byh-dense},
la restriction \`a~$k(X)$ de la structure de corps globalement valu\'ee de~$L$ 
est approch\'ee par une fonctionnelle divisorielle g\'eom\'etrique;
quitte \`a remplacer~$X$ par un mod\`ele birationnel, on peut
supposer qu'elle est associ\'ee \`a  un diviseur de Cartier ample~$A$ sur~$X$.
D'apr\`es le th\'eor\`eme de Bertini,
il existe une courbe~$C$ sur~$X$ qui n'est pas contenue dans~$Y$
et un nombre r\'eel~$w>0$ tel que dans la classe de $A^{d-1}$
soit \'egale \`a~$w[C]$.
Le corps~$k(C)$ est de type fini sur~$k$ et de degr\'e de transcendance~$1$,
il admet donc un $k$-plongement dans~$K$.

D'apr\`es le lemme suivant,
il existe un tel plongement qui induit~$w$ fois la structure
induite par~$L$. Le point g\'en\'erique de~$C$ 
fournit alors l'\'el\'ement $(c_1,\dots,c_n)\in K^n$  qui satisfait les 
conditions requises.

\begin{lemm}
Soit $k$ un corps et soit $K$ la cl\^oture alg\'ebrique de~$k(T)$,
consid\'er\'ee comme corps globalement valu\'e avec $\Ht(T)=1$. 
Soit $a$ un nombre rationnel strictement positif.
Il existe un $k$-automorphisme~$\theta$ de~$K$
tel que $\Ht(\theta (x))=a \Ht(x)$ pour tout $x\in K$.
\end{lemm}
\begin{proof}
Il suffit de traiter le cas o\`u $a$ est un entier naturel non nul.
Soit $\theta\colon k(T)\to K$ le $k$-morphisme tel que $\theta(T)=T^a$.
Il se prolonge en un $k$-morphisme de~$K$ dans lui-m\^eme,
encore not\'e~$\theta$; c'est un $k$-automorphisme de~$K$.
Par unicit\'e \`a scalaire pr\`es de la structure de corps globalement valu\'e sur~$K$
qui induit la structure triviale sur~$k$, il existe un nombre r\'eel~$c\geq0$
tel que $\Ht(\theta(x))=c \Ht(x)$ pour tout $x\in K$; en prenant $x=T$,
on voit que $c=a$, d'o\`u le lemme.
\end{proof}

Traitons maintenant le cas g\'en\'eral. 
Par un argument de th\'eorie des mod\`eles, il s'agit
de d\'emontrer l'existence d'un $K$-morphisme de~$L$ 
dans une ultrapuissance de~$K$.
L'\'etude pr\'ec\'edente des formules \`a param\`etres dans~$k$
entra\^{\i}ne l'existence d'un $k$-morphisme~$\sigma_L$ de~$L$
dans une ultrapuissance~$K^*$ de~$K$, associ\'ee \`a un ultrafiltre
non principal~$\mathfrak u$ sur un ensemble infini~$I$.
Consid\'erons la restriction~$\sigma$ de~$\sigma_L$ \`a~$k[T]$;
posons $r=\Ht(\sigma(T))$; on a $r>0$.
On peut repr\'esenter~$\sigma$ comme l'ultrapuissance d'une famille
$(\sigma_i)$ de $k$-morphismes de~$k[T]$ dans~$K$; 
posons $r_i=\Ht(\sigma_i(T))$.
Alors, $\lim_{i,\mathfrak u} r_i = r$.
Comme $r>0$, l'ensemble des~$i$ tels que $r_i>0$ appartient \`a~$\mathfrak u$;
il n'est pas restrictif de supposer que $r_i>0$ pour tout~$i$.
Alors, pour tout~$i$, on a $\sigma_i(T)\notin k$;
comme $k$ est alg\'ebriquement clos, 
le morphisme~$\sigma_i\colon k[T]\to K$ 
se prolonge en un $k$-morphisme de~$k(T)$ dans~$K$, 
puis, l'extension~$K$ de~$k(T)$ \'etant alg\'ebrique, 
en un $k$-morphisme de~$K$ dans~$K$ que l'on note encore~$\sigma_i$.

La description valuative des hauteurs montre 
que pour tout~$i$,
$r_i=\Ht(\sigma_i(T))$ est un nombre \emph{rationnel} strictement positif.
D'apr\`es le lemme pr\'ec\'edent, il existe pour tout~$i$ 
un automorphisme~$\theta_i$ de~$K$
tel que $\mathord{\Ht}\circ\theta_i = a_i \Ht$.
L'ultrapuissance des~$\theta_i$ d\'efinit un $k$-automorphisme~$\theta$
de l'ultrapuissance~$K^*$.
Alors, $\theta^{-1}\circ\sigma_L$ est un $K$-plongement de~$L$ dans~$K^*$,
ce qui conclut la d\'emonstration du th\'eor\`eme~\ref{theo.byh}.

\begin{rema}
La d\'emonstration a utilis\'e de mani\`ere cruciale 
l'unicit\'e \`a facteur pr\`es d'une structure de corps valu\'e
sur la cl\^oture alg\'ebrique de~$k(T)$ qui induise la structure
triviale sur~$k$.

Ce r\'esultat ne vaut pas pour la cl\^oture
alg\'ebrique d'un corps de fractions rationnelles $K=k(S,T)$
en deux variables. Munissons-le d'une structure de corps globalement
valu\'e triviale sur~$k$, par exemple d\'eduite d'un fibr\'e en droites
ample sur~$\P_2$ et soit $a,b\in K$ deux \'el\'ements alg\'ebriquement
ind\'ependants sur~$k$ tels que $\Ht(a)\neq \Ht(b)$.
Soit $L$ le sous-corps de~$K$ donn\'e par $L=k(a+b,ab)$
muni de sa structure de corps globalement valu\'e induite par celle de~$K$.
Comme $K$ est alg\'ebrique sur~$L$, on peut \'egalement le munir
de la structure de corps globalement valu\'e sym\'etrique qui prolonge
celle de~$L$; notons~$K'$ ce corps globalement valu\'e.
Puisque $\Ht(a)\neq\Ht(b)$, les structures de corps globalement
valu\'es de~$K$ et~$K'$ ne sont pas isomorphes. 
On peut m\^eme en d\'eduire que la cl\^oture alg\'ebrique de~$K$,
munie de la structure de corps globalement valu\'ee sym\'etrique qui 
prolonge celle de~$K$,
n'est pas existentiellement close.
\end{rema}
\subsection{Cl\^oture existentielle: corps de nombres}

\begin{theo}[\cite{Szachniewicz-2023}] \label{theo.s}
Munissons le corps~$\Q$ des nombres rationnels
de sa structure standard  de corps globalement valu\'e,
pour laquelle $h(1,2)=\log(2)$. 
Alors le corps~$\overline\Q$ des nombres alg\'ebriques,
muni de l'unique structure de corps globalement valu\'e
qui est sym\'etrique et \'etend la structure donn\'ee sur~$\Q$,
est existentiellement clos.
\end{theo}

Soit $F$ un corps globalement 
valu\'e de caract\'eristique z\'ero, induisant la structure standard sur~$\Q$,
et soit $b=(b_1,\dots,b_n)$ une suite finie d'\'el\'ements de~$F$; 
il s'agit d'approcher le type sans quantificateur de~$b$.
On peut remplacer~$F$ par le corps $\Q(b_1,\dots,b_n)$. Alors,
sa structure de corps globalement valu\'e se d\'ecrit en termes g\'eom\'etriques.
La question se traduit ainsi en un \'enonc\'e de g\'eom\'etrie arithm\'etique.

Soit $\mathscr X$ un sch\'ema propre sur~$\Spec(\Z)$, surjectif, int\`egre
et normal; soit $d=\degtr_\Q (\kappa(\mathscr X))$, de sorte
que $\dim(\mathscr X)=d+1$.

Tout diviseur arithm\'etique $(D,g_D)$ sur~$\mathscr X$
induit une fonction \emph{hauteur} sur l'ensemble des points ferm\'es 
de $\mathscr X_\Q$.
% \footnote{Peut-\^etre faire en g\'en\'eral
% lors de la d\'efinition de l'intersection arithm\'etique?}
Soit $y$ un tel point et soit~$\mathscr Y$
son adh\'erence dans~$\mathscr X$ pour la topologie de Zariski;
si $y$ n'appartient pas au support de~$D$, on peut restreindre~$D$
\`a~$\mathscr Y$ et poser
\[ h_{(D,g_D)}(y) = (D|_{\mathscr Y}) 
 = \sum_p \nu_p(D, \mathscr Y) \log (p) + \int_{\mathscr Y(\C)} g_D , \]
o\`u la somme porte sur l'ensemble des nombres premiers~$p$,
$\nu_p(D,\mathscr Y)$ est la multiplicit\'e d'intersection au-dessus de~$(p)$
du 1-cycle~$\mathscr D$ et du diviseur de Cartier~$D$ sur~$\mathscr X$,
et l'int\'egrale $\int_{\mathscr Y(\C)}g_D$ est la somme des valeurs
de~$g_D$ en tous les conjugu\'es de~$y$ dans~$\mathscr X(\C)$.
Dans le cas g\'en\'eral, on remplace~$(D,g_D)$ par un diviseur arithm\'etique
lin\'eairement \'equivalent $(D',g_{D'})$ tel que $y$ n'appartienne
pas au support de~$D'$.

Pour tout point~$y$, l'application $(D,g_D)\mapsto h_{(D,g_D)}(y)$
est lin\'eaire. Elle s'annule en tout diviseur arithm\'etique
de la forme $\hdiv(f)$.

\begin{prop}
Soit $\lambda\colon \hDiv(\mathscr X)\to\R$ une forme lin\'eaire
qui est positive sur les diviseurs arithm\'etiques pseudo-effectifs
et v\'erifie $\lambda((0,1))=1$.
Soit $Z$ un ferm\'e de Zariski strict de~$\mathscr X_\Q$.
Alors $\lambda$ est limite de formes lin\'eaires de la forme
$(D,g_D)\mapsto h_{(D,g_D)}(y)$, o\`u $y$ parcourt l'ensemble
des points alg\'ebriques de~$\mathscr X_\Q\setminus Z$.
\end{prop}
\begin{proof}
Quitte \`a remplacer~$\mathscr X$ par un mod\`ele birationnel,
le th\'eor\`eme~\ref{theo.s-dense} permet de supposer que $\lambda$
est de la forme $(D,g_D)\mapsto \frac1{\vol(A_\Q)}\langle (A,g_A)^d \rangle\cdot (D,g_D)$,
o\`u $(A,g_A)$ est un diviseur arithm\'etique ample sur~$\mathscr X$.
Le th\'eor\`eme suivant affirme alors le r\'esultat voulu.
\end{proof}

\begin{theo}[\cite{Szachniewicz-2023}] \label{theo.s-h}
Soit $(A,g_A)$ un diviseur arithm\'etique ample sur~$\mathscr X$
et soit $Z$ un ferm\'e de Zariski strict de~$\mathscr X_\Q$.
Soit $(D_i,g_i)_i$ une famille finie de diviseurs arithm\'etiques
sur~$\mathscr X$. 
Pour tout $\eps>0$, il existe un point ferm\'e~$y\in \mathscr X_\Q$
tel que $y\notin Z$ et tel que 
\[ \Abs{ h_{(D_i,g_i)} (y) - \frac1{\vol (A_\Q)} (A,g_A)^d (D_i,g_i)}<\eps \]
pour tout~$i$.
\end{theo}

\begin{proof}
La g\'eom\'etrie d'Arakelov a largement exploit\'e l'interaction
entre hauteurs et intersection arithm\'etique,  en particulier
l'in\'egalit\'e fondamentale, pour tout diviseur arithm\'etique ample $(D,g_D)$,
 \begin{equation}
\label{eq.zhang} 
   \liminf_{y\in \mathscr X_\Q} h_{(D,g_D)}(y) \geq  \frac{(D,g_D)^{d+1}}{(d+1)\vol(D_\Q)}, \end{equation}
d\'emontr\'ee par~\textcite{Zhang-1995b}.
La limite inf\'erieure signifie que $y$ tend vers le point g\'en\'erique de~$\mathscr X_\Q$.
La d\'emonstration du th\'eor\`eme~\ref{theo.s-h}
reprend celle de cette in\'egalit\'e 
en la pr\'ecisant en plusieurs points cruciaux.

Le th\'eor\`eme d'amplitude arithm\'etique de~\textcite{Zhang-1995b}
fournit, lorsque $n$ tend vers l'infini,
des sections $s\in \widehat H^0(\mathscr X, n(D,g_D))$
qui ne s'annulent pas en des points prescrits.
C'est la cl\'e de l'in\'egalit\'e pr\'ec\'edente.

Le th\'eor\`eme 1.5 de~\textcite{Charles-2021} affirme que pour tout
ferm\'e strict~$\mathscr Y$ de~$\mathscr X$, la proportion 
de ces sections~$s$ telles que $\div(s)$ ne soit pas contenu
dans~$\mathscr Y$ tend vers~1.

Lorsque $d\geq 1$, \textcite{Wilms-2022} et~\textcite{QuYin-2024}
ont \'etabli que lorsque $\eps\to 0$,
qu'une proportion arbitrairement proche de~1 de ces sections~$s$
v\'erifient les conditions suivantes:
\begin{itemize}
\item Le diviseur~$\div(s)$ est irr\'eductible, sa fibre g\'en\'erique est lisse;
\item La fonction positive $\log \norm s^{-1}$ sur $\mathscr X(\C)$
est petite en norme~$L^1$ par rapport \`a la mesure $\omega_{(D,g_D)}^d$
sur~$\mathscr X(\C)$;
plus pr\'ecis\'ement:
\[ \int_{\mathscr X(\C)} \log \norm s^{-1} \omega_{(D,g_D)}^d < \eps (D,g_D)^{d+1}. \]
\end{itemize}
Si $d\geq 1$, une variante de ce dernier \'enonc\'e permet 
de choisir une section $s\in\widehat H^0(\mathscr X, n(D,g_D))$
telle que $\div(s)$ soit irr\'eductible, sa fibre g\'en\'erique soit lisse
et non contenue dans~$Z$,
et telle que $\log \norm s^{-1}$ soit petite en norme~$L^1$
par rapport aux mesures $\omega_{D_i,g_i}\wedge \omega_{A,g_A}^{d-1}$.
Posons $\mathscr Y=\div(s)$. 
L'assertion pour la normalisation de~$\mathscr Y$ 
et les restrictions \`a des diviseurs arithm\'etiques $(D_i,g_i)$
fait intervenir
$(A,g_A)|_{\mathscr Y}^{d-1}(D_i,g_{D_i})|_{\mathscr Y}$
qui diff\`ere de $(A,g_A)^d(D_i,g_{D_i})$ par un terme archim\'edien,
int\'egrale de $\log\norm s^{-1}$
qui, par le choix de~$s$, est petit. Elle entra\^{\i}ne ainsi
l'assertion pour~$\mathscr X$.
Par r\'ecurrence, l'assertion se ram\`ene ainsi au cas \'evident o\`u $d=0$.
\end{proof}

\subsection{Perspectives}

Comme \'evoqu\'e dans l'introduction, les r\'esultats expos\'es dans
ce texte ouvrent plus encore de questions qu'ils n'en r\'esolvent;
ils offrent aussi un regard neuf sur des r\'esultats
de g\'eom\'etrie diophantienne, certains anciens.

Sous sa version initiale, le th\'eor\`eme de~\textcite{FeketeSzego-1955}
concerne l'existence d'entiers alg\'ebriques dont tous les conjugu\'es appartiennent
\`a une partie~$S$ du plan complexe qui est compacte et sym\'etrique
par rapport \`a l'axe r\'eel. Si la capacit\'e logarithmique de cette
partie est strictement n\'egative,  
il n'y en a qu'un nombre fini ; si elle est positive,
tout voisinage en contient une infinit\'e. Si elle est nulle,
on peut ajouter que ces points (et leurs conjugu\'es) s'\'equidistribuent
vers la mesure d'\'equilibre de cette partie \parencite[th\'eor\`eme~1.2.6]{Serre-2019a}.

Ces \'enonc\'es admettent des variantes sur des vari\'et\'es alg\'ebriques
de dimension arbitraire, et la th\'eorie des corps globalement valu\'es 
les \'eclaire d'un jour nouveau: ces th\'eor\`emes deviennent un simple
cas particulier du th\'eor\`eme~\ref{theo.s}
que $\overline\Q$  est un corps globalement valu\'e existentiellement clos.
Cependant, les d\'emonstrations ne sont pas tout \`a fait nouvelles, car 
les techniques mises en \oe uvre sont parall\`eles 
\`a celles des th\'eor\`emes initiaux, ou de leurs extensions via
la g\'eom\'etrie d'Arakelov et l'utilisation de th\'eor\`emes
comme l'in\'egalit\'e~\eqref{eq.zhang}.

La th\'eorie des corps globalement valu\'es pose \'egalement de nouvelles questions.
Si $K$ est un corps globalement valu\'e, on peut consid\'er\'e le $1$-type~$p$
des points de hauteur~$0$. Les r\'esultats que je viens d'\'evoquer
donnent acc\`es \`a sa partie \og sans quantificateurs \fg~$p_{\text{sq}}$.

Dans le cas purement non archim\'edien, 
les \'el\'ements~$x$ de~$K$ tels que $\Ht(x)=0$ forment
alors un sous-corps~$k$ de~$K$ et ce type~$p_{\text{sq}}$ se d\'ecrit en termes
du type g\'en\'erique sur~$k$.
Dans le cas arithm\'etique, lorsque $K=\Q$
est muni de sa structure standard, 
le th\'eor\`eme d'\'equidistribution de~\textcite{Bilu-1997}
s'incarne de la fa\c{c}on suivante: 
ce type sans quantificateur 
est constitu\'e de l'origine, des racines de l'unit\'e et d'un unique 
type sans quantificateur auxiliaire. 
En ins\'erant dans la d\'emonstration les arguments du th\'eor\`eme~1.2.6
de~\textcite{Serre-2019a}, 
on obtient 
d'ailleurs un r\'esultat
analogue pour la fonction hauteur associ\'ee \`a un compact de~$\C$
par la th\'eorie du potentiel en rempla\c{c}ant $\sup(1,\log(x))$
par la fonction de Green de ce compact.

De m\^eme, ces travaux ouvrent la voie \`a une \'etude 
des hauteurs sur les vari\'et\'es ab\'eliennes sur les corps globalement valu\'es,
parall\`ele \`a celle qui est faite sur les corps de type fini
et \`a la recherche d'analogues des th\'eor\`emes classiques,
comme celui (\textsc{Lang}, \textsc{N\'eron})
qui d\'ecrit les points de hauteur nulle sur une vari\'et\'e ab\'elienne
d\'efinie sur un corps de type fini. 

Ces travaux ouvrent enfin de nombreuses questions
de th\'eorie des mod\`eles: stabilit\'e, rangs, imaginaires\dots

%% printbibliography is the command from the package biblatex
 
% \printshorthands %%% to remove in case you would not like to use shorthand

\printbibliography

\addressindent.45\textwidth

\end{document}